\documentclass[letter]{article}
\usepackage{amsmath}
\usepackage[utf8]{inputenc}
\usepackage[normalem]{ulem}

\title{Bow Metrics and Hyperbolicity}
\author{Feodor F. Dragan\thanks{ Kent State University, Computer Science Department, Kent, Ohio, USA}
  \and Guillaume Ducoffe\thanks{University of Bucharest, Faculty of Mathematics and Computer Science, and National Institute for Research and Development in Informatics, Romania}
  \and Michel Habib\thanks{IRIF, CNRS \& Université Paris Cité, Paris, France}
  \and Laurent Viennot\thanks{Inria, DI ENS, Paris, France}}

\usepackage{a4wide}
\usepackage{color}
\usepackage{graphicx}
\usepackage{xspace}
\usepackage{url}
\usepackage{amsmath,amssymb}
\usepackage{hyperref}
\hypersetup{pdfstartview=XYZ}

\usepackage{capt-of}

\newcommand{\commentout}[1]{}

\usepackage[vlined,english,boxed]{algorithm2e} 
      \SetAlgoLined
      \SetKwInput{Input}{Input}
      \SetKwInput{Output}{Output}
      \SetKw{Return}{Return}
      \SetKwIF{If}{ElseIf}{Else}{If}{then}{else if}{else}{endif}
      \SetKwFor{For}{For}{do}{done}
      \SetKwFor{for}{For}{}{}
      \SetKwFor{While}{While}{do}{done}
      \SetKwFor{Repeat}{Repeat}{}{}
      \SetKwRepeat{Do}{Do}{while}
      \SetKwFor{Procedure}{Procedure}{}{}
      \SetKwFor{Function}{Function}{}{}
      \SetKw{Return}{Return}%
      \SetKw{return}{return}%
      \SetKwComment{Comment}{/* }{ */}
      \SetSideCommentRight
      \DontPrintSemicolon

\usepackage{todonotes}

\long\def\jump#1\finjump{}
\long\def\beglongversion#1\endlongversion{#1}

\newtheorem{remark}{Remark}
\newtheorem{theorem}{Theorem}
\newtheorem{lemma}{Lemma}
\newtheorem{corollary}{Corollary}
\newtheorem{proposition}{Proposition}

\def\Box{\hbox{\hskip 1pt \vrule width 4pt height 8pt depth 1.5pt \hskip 1pt}}
\newenvironment{proof}{\medskip\noindent\textbf{Proof.}}{{}\hfill$\Box$\\}

\usepackage{amsfonts}

\newtheorem{claim}{Claim}

\begin{document}

\maketitle

\begin{abstract} A ($\lambda,\mu$)-bow metric was defined in (Dragan \& Ducoffe, 2023) as a far reaching generalization of an $\alpha_i$-metric (which is equivalent to a ($0,i$)-bow metric). A graph $G=(V,E)$ is said to satisfy 
{\em($\lambda,\mu$)}-bow metric if for every four vertices $u,v,w,x$ of $G$ the following holds:  if two shortest paths $P(u,w)$ and $P(v,x)$ share a common shortest subpath $P(v,w)$  of length more than $\lambda$ (that is, they overlap by more than  $\lambda$),  then the distance between $u$ and $x$ is at least $d_G(u,v)+d_G(v,w)+d_G(w,x)-\mu$. ($\lambda,\mu$)-Bow metric can also be considered for all geodesic metric spaces. It was shown by Dragan \& Ducoffe that every $\delta$-hyperbolic graph (in fact, every  $\delta$-hyperbolic geodesic metric space)  satisfies ($\delta, 2\delta$)-bow metric. Thus, ($\lambda,\mu$)-bow metric  is a common generalization of hyperbolicity and of $\alpha_i$-metric. 
In this paper, we investigate an intriguing question whether ($\lambda,\mu$)-bow metric implies hyperbolicity in graphs. Note that, this is not the case for general geodesic metric spaces as Euclidean spaces satisfy  ($0,0$)-bow metric whereas they have unbounded hyperbolicity. We conjecture that, in graphs, ($\lambda,\mu$)-bow metric indeed implies hyperbolicity and show that our conjecture is true for several large families of graphs. 
  \medskip

\end{abstract}

\section{Introduction} 
$\delta$-Hyperbolic metric spaces have been defined by M. Gromov \cite{Gromov1987} in 1987 via a simple 4-point condition: for any four
points $u, v, w, x$, the two larger of the distance sums $d(u, v)+
d(w, x), d(u, w) + d(v, x), d(u, x) + d(v, w)$ differ by at most
$2\delta$. They play an important role in geometric group theory, geometry of negatively curved spaces \cite{ABCFLMSS1991,GH-book,Gromov1987}, and have become of interest in several domains of computer science (see. e.g., \cite{AbDr16,ADM14,BoChCa15,delta-hyp-1st,CDEHVX12,Chepoi2018FastEA,ChepoiEst,slimness,Hyp-p-center,Eppstein07,GaLy05,JoLo04,KrLee06,KSN16,Vien,ShTa04,VSuri} and subsection on related work for some details). $\delta$-Hyperbolicity measures, to some extent, the deviation of a metric from a tree metric. Recall that a metric space $(X, d)$ embeds into a tree network (with positive real edge lengths), that is, $d$ is a tree metric, if and only if for any four points $u, v, w, x,$ the two larger of the distance
sums $d(u, v) + d(w, x), d(u, w) + d(v, x), d(u, x) + d(v, w)$ are equal. 
A connected graph $G = (V, E)$ equipped with standard graph metric $d_G$ is $\delta$-hyperbolic if the metric space $(V, d_G)$ is $\delta$-hyperbolic. The smallest value $\delta$ for which $G$ is $\delta$-hyperbolic is called the hyperbolicity $\delta(G)$ of $G$. 

For graphs, in 1986, V. Chepoi introduced a notion of $\alpha_i$-metric \cite{Ch1986-ch,Ch88} via another simple 4-point condition: for any four
vertices $u, v, w, x$, if a shortest path between $u$ and $w$ and a shortest path between $x$ and $v$ share a terminal edge $vw$, then $d_G(u,x)\geq d_G(u,v) + d_G(v,x)-i$. Roughly, gluing together any two shortest paths along a common terminal edge may not necessarily result in a shortest path but yields a ``near-shortest'' path with defect at most $i$. 
A graph is called $\alpha_i$-metric if it satisfies the $\alpha_i$-metric for every four vertices $u,v,w,x$. 
Evidently, every graph is an $\alpha_i$-metric graph for some $i$. Furthermore, several known graph classes are $\alpha_i$-metric for some small values of $i$ (see \cite{Ch1986-ch,Ch88,HHD-03,alpha-centers,alpha-hyperb,WG16,Howorka,YuCh1991}  and subsection on related work). 

Little was known until recent paper \cite{alpha-hyperb} about the relationships between the $\alpha_i$-metric  and the hyperbolicity. Earlier, the authors of~\cite{BaCh-survey} observed that every $0$-hyperbolic graph must be $\alpha_0$-metric, and every $\frac 1 2$-hyperbolic graph must be $\alpha_1$-metric. 
The authors of~\cite{ChChChJa22+} briefly discussed the  hyperbolicity and the $\alpha_i$-metric for geodesic metric spaces (i.e., metric spaces where any two points can be connected by a geodesic/shortest path).  They observed that Euclidean spaces must be $\alpha_0$-metric (because the union of two geodesics, i.e., line segments, that slightly overlap must be a geodesic, i.e., a line segment) whereas they have unbounded hyperbolicity. They also noted that ``for graphs, the links between $\delta$-hyperbolic graphs and graphs with $\alpha_i$-metrics are less clear''. 

Those missing links for graphs were recently clarified in \cite{alpha-hyperb}. It was shown that every  $\alpha_i$-metric graph is $\delta$-hyperbolic for some $\delta\le \frac{3}{2}(i+1)$. Furthermore, the authors of \cite{alpha-hyperb} proved that every  $\alpha_1$-metric graph is 1-hyperbolic and conjectured that the right upper bound for the hyperbolicity of an $\alpha_i$-metric graph might be $f(i) = \frac{i+1}{2}$, that would be sharp. They also noted  that, for any positive integer $i$, there exists a $1$-hyperbolic graph that is not $\alpha_i$-metric; the $(1\times n)$-rectilinear grid has hyperbolicity 1 and is $\alpha_i$-metric only for $i=\Omega(n)$. Paper \cite{alpha-hyperb} additionally proposed a far reaching generalization of the $\alpha_i$-metric. This new metric, called in \cite{alpha-hyperb}  ($\lambda,\mu$)-bow metric, generalizes also the hyperbolicity. A graph $G=(V,E)$ is said to satisfy 
{\em($\lambda,\mu$)}-bow metric if for every four vertices $u,v,w,x$ of $G$ the following holds: 
\begin{description} 
\item[($\lambda,\mu$)-bow metric:] if two shortest paths $P(u,w)$ and $P(v,x)$ share a common shortest \\ subpath $P(v,w)$  
of length more than $\lambda$ (that is, they overlap by more than  $\lambda$),  then \\ the distance between $u$ and $x$ is at least $d_G(u,v)+d_G(v,w)+d_G(w,x)-\mu$. 
\end{description} 

\noindent 
 Clearly, $\alpha_i$-metric graphs satisfy ($0,i$)-bow metric. 
Furthermore, this generalization is more robust to some graph operations. For instance, the $1$-subdivision of a ($\lambda,\mu$)-bow metric graph must satisfy ($2\lambda+2,2\mu$+2)-bow metric (see Lemma \ref{lm:1-subdiv}). 
This notion of ($\lambda,\mu$)-bow metric can also be considered for all geodesic metric spaces. In \cite{alpha-hyperb}, it was shown that every $\delta$-hyperbolic graph (in fact, every  $\delta$-hyperbolic geodesic metric space)  satisfies ($\delta, 2\delta$)-bow metric. Thus, ($\lambda,\mu$)-bow metric  is a common generalization of hyperbolicity and of $\alpha_i$-metric. We believe that the study of ($\lambda,\mu$)-bow metrics could help in deriving new properties of $\alpha_i$-metric graphs and $\delta$-hyperbolic graphs. 

In this paper, we investigate an intriguing question whether ($\lambda,\mu$)-bow metric implies hyperbolicity in graphs. Note that, this is not the case for general geodesic metric spaces as Euclidean spaces satisfy  ($0,0$)-bow metric whereas they have unbounded hyperbolicity. We conjecture that, in graphs, ($\lambda,\mu$)-bow metric indeed implies hyperbolicity and show that our conjecture is true for several large families of graphs.

     \paragraph{Additional related work on $\alpha_i$-metrics and hyperbolicity.}

The $\alpha_i$-metric property was introduced by V. Chepoi in~\cite{Ch1986-ch,Ch88} and was further investigated in~\cite{HHD-03,WG16,YuCh1991}. It was shown that all chordal graphs~\cite{Ch1986-ch} and all plane triangulations with inner vertices of degree at least seven~\cite{WG16} are $\alpha_1$-metric. 
All distance-hereditary graphs~\cite{YuCh1991}, and even more strongly, all HHD-free graphs~\cite{HHD-03}, are $\alpha_2$-metric. The $\alpha_0$-metric graphs are exactly the ptolemaic graphs, {\it i.e.} the chordal distance-hereditary graphs~\cite{Howorka}. Chepoi and Yushmanov in \cite{YuCh1991} also provided a characterization of all $\alpha_1$-metric graphs. They are exactly the graphs where all balls are convex and a specific isometric subgraph is forbidden.  
Recently, additional properties of $\alpha_1$-metric graphs and $\alpha_i$-metric graphs ($i \in {\cal N}$) were reported in~\cite{alpha-centers,alpha-hyperb,DrGu21,WG16}.
In~\cite{alpha-centers}, the first algorithmic applications of $\alpha_i$-metric graphs to classical distance problems, such as diameter, radius and all eccentricities computations, were presented. 
More specifically, all vertex eccentricities in an $\alpha_i$-metric graph can be approximated in linear time up to some additive term in ${\cal O}(i)$. Furthermore, there exists a subquadratic-time algorithm for exact computation of the radius of an $\alpha_1$-metric graph.

Hyperbolicity was introduced by Gromov in his study on automatic groups~\cite{Gromov1987}. 
Since then, the study of $\delta$-hyperbolic graphs has become an important topic in Metric Graph Theory~\cite{BaCh-survey}.
This parameter has attracted further attention in Network Science, both as a way to better classify complex networks~\cite{AbDr16,KSN16} and to explain some of their properties such as core congestion~\cite{Chepoi:2017:CCI:3039686.3039835}. Many real-world networks have small hyperbolicity~\cite{AbDr16,KSN16,Vien}.  
See also~\cite{ADM14,BoChCa15,slimness,JoLo04} for other related results on the hyperbolicity. 
Several approaches have been proposed in order to upper bound the hyperbolicity in some graph classes~\cite{obstructions,KoMo02,DBLP:journals/combinatorics/WuZ11}. 
In particular, chordal graphs are $1$-hyperbolic, and the chordal graphs with hyperbolicity strictly less than one can be characterized with two forbidden isometric subgraphs~\cite{uea21813}.
The $0$-hyperbolic graphs are exactly the block graphs, {\it i.e.} the graphs such that every biconnected component is a clique~\cite{Hov79}.
Characterizations of $\frac 1 2$-hyperbolic graphs were given in~\cite{BaCh-1-hyp,CoDu14}.
Furthermore, the algorithmic applications of $\delta$-hyperbolic graphs have been studied much earlier than for $\alpha_i$-metric graphs~\cite{DBLP:journals/dcg/ChalopinCDDMV21,delta-hyp-1st,CDEHVX12,Chepoi2018FastEA,Hyp-p-center,GaLy05,KrLee06,ShTa04}.  
In~\cite{delta-hyp-1st,Chepoi2018FastEA} it was proved that all vertex eccentricities in a $\delta$-hyperbolic graph can be approximated in linear time up to some additive term in ${\cal O}(\delta)$. 

    \paragraph{Our contributions.}
We prove that our conjecture that {\em ($\lambda,\mu$)-bow metric implies hyperbolicity} is true for major graph classes known in Metric Graph Theory, namely, for meshed graphs, for graphs with convex balls, and for generalized Helly graphs (see Section \ref{Sec:imply-hyp}). Note that the class of meshed graphs alone contains basis graphs of matroids and of even $\Delta$-matroids,   1-generalized Helly graphs, weakly modular graphs, as well as the graphs in which all median sets induce connected or isometric subgraphs. Weakly modular graphs in turn contain such classical graph classes from Metric Graph Theory as modular graphs, pseudo-modular graphs, pre-median graphs, weakly median graphs, quasi-median graphs, dual polar graphs, median graphs, distance-hereditary graphs, bridged graphs, Helly graphs, chordal graphs, dually chordal graphs and  many others (see survey~\cite{BaCh-survey} for Metric Graph Theory and definitions of those graph classes). In Section \ref{Sec:imply-hyp}, we show also that our conjecture is true for all graphs in which side lengths of metric triangles are bounded or interval thinness is bounded. In Section \ref{Sec:graphclasses}, as a warm up, 
we demonstrate that many known in literature graph classes satisfy $(\lambda,\mu)$-bow metric for some bounded values of $\lambda$ and $\mu$. Additionally to graphs with hyperbolicity $\delta$, precise bounds are given for graphs with bounded slimness, with bounded tree-length, $k$-chordal graphs, AT-free graphs, and others  (see Table \ref{table:hb_sl_bow}).

\begin{table} [htbp]
	\centering
	\begin{tabular}{|l|l|l|l|}
		\hline
		Graph Class                & Hyperbolicity & Slimness & Bow Metric\\ \hline\noalign{\smallskip}
  graphs with hyperbolicity $\delta$ & $=\delta$ & $\le 3\delta+\frac{1}{2}$ \cite{approx-hb}  & ($\delta, 2\delta$)-bow \cite{alpha-hyperb}\\ \hline
\noalign{\smallskip}	
 graphs with slimness $\varsigma$ & $\le 2\varsigma +\frac{1}{2}$ \cite{Soto} & $=\varsigma$ & ($\varsigma, 2\varsigma$)-bow  [here]\\ \hline	
\noalign{\smallskip}	
		graphs with tree-length $\lambda$ & $\leq\lambda$\cite{{delta-hyp-1st}} & $\leq\lfloor\frac{3}{2}\lambda\rfloor$ \cite{diestel2012connectedTW,slimness} & ($\lambda, 2\lambda$)-bow  [here] \\ \hline\noalign{\smallskip}	
  k-chordal graphs ($k\ge 4$) & $\leq\lfloor\frac{k}{2}\rfloor/2$\cite{DBLP:journals/combinatorics/WuZ11} &$\leq\lfloor\frac{k}{4}\rfloor+1$ \cite{bermudo2016hyperbolicity,slimness} & ($\lfloor\frac{k}{4}\rfloor, \lfloor\frac{k}{2}\rfloor$)-bow  [here] \\ \hline
\noalign{\smallskip}	
		AT-free graphs     & $\leq1$   \cite{DBLP:journals/combinatorics/WuZ11}          & $\leq1$ \cite{slimness}     &  ($1,2$)-bow  [here] \\ \hline
\noalign{\smallskip}	
		HHD-free  graphs    &      $\leq1$  \cite{DBLP:journals/combinatorics/WuZ11}             & $\leq1$  \cite{slimness}     & ($0,2$)-bow \cite{HHD-03}  \\  \hline		
\noalign{\smallskip}	  
  Chordal graphs            & $\leq1$ \cite{uea21813}            & $\leq1$ \cite{slimness}     & ($0,1$)-bow \cite{Ch1986-ch,Ch88}  \\  \hline
\noalign{\smallskip}	
		              Distance-Hereditary graphs & $\leq1$ \cite{DBLP:journals/combinatorics/WuZ11}               &$\leq1$  \cite{slimness}   &   ($0,2$)-bow \cite{YuCh1991}   \\ \hline
\noalign{\smallskip}	
		              Ptolemaic graphs & $\leq1/2$ \cite{BaCh-1-hyp}               &$\leq1$  \cite{slimness}   &   ($0,0$)-bow \cite{Howorka}   \\ \hline     \end{tabular}
	\caption{Hyperbolicity, slimness and bow metric for various  graph classes.}
	\label{table:hb_sl_bow}
\end{table}

  \section{Preliminaries} 
For any undefined graph terminology, see~\cite{BoMu08}.
In what follows, we only consider graphs $G=(V,E)$ that are finite, undirected, unweighted, simple and connected. 

The {\em distance} $d_G(u,v)$ between two vertices $u,v \in V$ is the minimum length (number of edges) of a path between $u$ and $v$ in $G$. The {\em ball} with center $v$ and radius $r$ is defined as $B(v,r) = \{u \in V : d_G(u,v) \le r\}$. For any subset $S\subseteq V$, let $B(S,r)= \{u \in V : d_G(u,S)\le r\}$, where $d_G(u,S):=\min\{d_G(u,v): v\in S\}$,  denote the ball around $S$ with radius $r$ (equivalently, the {\em $r$-neighborhood} of $S$).  

The {\em interval} $I_G(u,v)$ between $u$ and $v$ contains every vertex on a shortest $(u,v)$-path, {\it i.e.}, $I_G(u,v) = \{ w \in V : d_G(u,v) = d_G(u,w) + d_G(w,v)\}$.
Let also $I_G^o(u,v)=I_G(u,v)\setminus \{u,v\}$. We will omit the subscript if $G$ is clear from the context. 
For every integer $k$ such that $0 \le k \le d(u,v)$, we define the {\em slice} $S_k(u,v) = \{x \in I(u,v) : d(u,x) = k\}$.
Let $\tau(u,v)$ denote the maximum diameter of a slice between $u$ and $v$: $\tau(u,v) = \max_{0 \le k \le d(u,v)}\max\{d_G(x,y): x,y \in S_k(u,v)\}$.
We define the {\em interval thinness} of $G$ as $\tau(G) = \max\{\tau(u,v) : u,v \in V\}$. We say that the intervals of a graph $G$
are {\em $p$-thin} if $\tau(G)\le p$. 
It is easy to prove that every $\delta$-hyperbolic graph has interval thinness at most $2\delta$.
Conversely, odd cycles are examples of graphs with interval thinness equal to zero (they are so-called ``geodetic graphs'') but unbounded hyperbolicity.
However, let the {\em $1$-subdivision} graph $\Sigma(G)$ of a graph $G$ be obtained by replacing all its edges $e=uv$ by internally vertex-disjoint paths $[u,e,v]$ of length two.
Papasoglu~\cite{Pap95} proved that the hyperbolicity of $G$ is at most doubly exponential in the interval thinness of $\Sigma(G)$.


Three vertices $x$, $y$, and  $z$ of a graph $G$ form a {\em metric
triangle} $xyz$ if the intervals $I(x,y)$, $I(y,z)$ and $I(x,z)$ pairwise intersect only
in the common end vertices. The pairs $xy$, $xz$, and $yz$  are called 
the {\em sides} of  $xyz$. The integer $k:=\max\{d(x,y),d(y,z),$ $d(z,x)\}$ is called the {\em  maximum side-length} of the triangle. If $d(x,y)= d(y,z) = d(z,x) = k$, then this metric triangle is called {\em equilateral of size} $k$. A metric triangle $xyz$ has type ($k_1,k_2,k_3$) if its sides have lengths $k_1$, $k_2$, $k_3$ and
$k_1\ge k_2\ge k_3$.

Given a triple $u,v,w$, a {\em quasi-median} of $u,v,w$ is a metric triangle $u'v'w'$ such that:
  \begin{itemize}
   \item[]  $d(u,v) = d(u,u') + d(u',v') + d(v',v);$ 
   \item[]  $d(v,w) = d(v,v') + d(v',w') + d(w',w);$
  \item[]  $d(w,u) = d(w,w') + d(w',u') + d(u',u).$
\end{itemize}
If $u', v',$ and $w'$ are the same vertex $z$, or equivalently, if the size 
of $u'v'w'$ is zero, then this vertex $z$ is called a median of $u,v,w$. A median may not exist and may not be unique. On the other hand, every triple has at least one quasi-median ({\it e.g.}, see~\cite{BaCh-survey}).

A {\em geodesic triangle} $\Delta(u,v,w) = P(u,v) \cup P(v,w) \cup P(w,u)$ is the union of a shortest $(u,v)$-path $P(u,v)$, a shortest $(v,w)$-path $P(v,w)$ and a shortest $(w,u)$-path $P(w,u)$. Note that $P(u,v),P(v,w),P(w,u)$ are called the {\em sides of the triangle}, and they may not be disjoint. 
A geodesic triangle $\Delta(u,v,w)$ is called {\em $\varsigma$-slim} if for any vertex $x\in V$ on any side $P(u,v)$ the distance from $x$ to $P(u,w)\cup P(w,v)$ is at most $\varsigma$, i.e., each path is contained in the union of the $\varsigma$-neighborhoods of the two others. A graph $G$ is called {\em $\varsigma$-slim}, if all geodesic triangles in $G$ are $\varsigma$-slim. The smallest value $\varsigma$ for which $G$ is $\varsigma$-slim is called the {\em slimness} $\varsigma(G)$ of $G$. It is known that the hyperbolicity of a graph and its slimness are within constant factors from each other. 

\begin{proposition} [\cite{approx-hb,Soto}] \label{prop:sl-vs-hb} The following inequalities are true between the hyperbolicity $\delta$ and the slimness $\varsigma$ of a graph: $\delta \le 2\varsigma + \frac{1}{2}$ and   $\varsigma\le 3\delta+\frac{1}{2}$. 
\end{proposition} 

Using the notion of interval, we can reformulate the definition of a {\em($\lambda,\mu$)}-bow metric. A graph $G=(V,E)$ is said to satisfy 
{\em($\lambda,\mu$)}-bow metric if for every four vertices $u,v,w,x$ of $G$ the following holds: 
\begin{description} 
\item[($\lambda,\mu$)-bow metric:] if $v\in I(u,w), w\in I(v,x)$ and $d(v,w)>\lambda$,  then \\ $~~~~~~~~~~~~~~~~~~~~~~d(u,x)\ge d(u,v)+d(v,w)+d(w,x)-\mu$. 
\end{description} 

Other notations, terminology, and important graph classes are locally defined at appropriate places throughout the paper.

\section{Graph classes and 
($\lambda,\mu)$-bow metric} \label{Sec:graphclasses}
 Since $\alpha_i$-metric graphs are exactly the graphs satisfying  ($0,i$)-bow metric, we know that ptolemaic graphs (i.e., graphs which are chordal as well as distance-hereditary)  are exactly the graphs with ($0,0$)-bow metric (see \cite{Howorka}), chordal graphs satisfy ($0,1$)-bow metric (see \cite{Ch1986-ch,Ch88}), HHD-free graphs as well as distance-hereditary graphs satisfy ($0,2$)-bow metric (see \cite{HHD-03} and \cite{YuCh1991}). The graphs with ($0,1$)-bow metric were characterized in \cite{YuCh1991} as the graphs with convex balls and without one specific isometric subgraph. Recall that, a {\em distance-hereditary graph} is a graph where each induced path is a shortest path, a graph $G$ is {\em chordal} if every induced cycle of $G$ has length 3.  A graph containing no induced domino or house or a hole (i.e., an induced cycle of length greater that 4) is called a {\em House-Hole-Domino--free} graph (HHD-free graph, for short). 
 
 We can show that various additional graph classes satisfy a ($\lambda,\mu$)-bow metric for some small values of $\lambda$ and $\mu$. 
First we formulate an important result from \cite{alpha-hyperb} and prove its analog 
for graphs with slimness $\varsigma$. 

\begin{proposition} [\cite{alpha-hyperb}] \label{prop:bow-metric}
Every $\delta$-hyperbolic graph and, generally, every   $\delta$-hyperbolic geodesic metric space satisfies ($\delta, 2\delta$)-bow metric. 
\end{proposition}

From Proposition \ref{prop:bow-metric} and Proposition \ref{prop:sl-vs-hb}, it already follows that  the graphs with slimness $\varsigma$ satisfy a bow metric. Below, we give a direct proof with sharper bounds. 

\begin{proposition}  \label{prop:bow-metric-sl}
Every graph and, generally, every geodesic metric space whose geodesic triangles are $\varsigma$-slim satisfies ($\varsigma, 2\varsigma$)-bow metric. 
\end{proposition}

\begin{proof} 
Consider arbitrary four vertices $u,v,w,x$ such that $v\in I(u,w), w\in I(v,x)$ and $d(v,w)>\varsigma$. We need to show that $d(u,x)\ge d(u,v)+d(v,w)+d(w,x)-2\varsigma.$  Consider a geodesic triangle  $\Delta(u,v,x)$ of $G$ formed by shortest paths $P(u,x)$,
$P(u,v)$, and $P(v,x)$ containing vertex $w$. We know that $d(u,w)=d(u,v)+d(v,w)>d(u,v)+\varsigma$. If there is a vertex $w'\in P(u,v)$ such that $d(w,w')\le \varsigma$, then $d(u,w)\le d(u,w')+d(w',w)\le d(u,v)+\varsigma$, contradicting with $d(u,w)>d(u,v)+\varsigma$. Thus, there must exist a vertex $y$ in $P(u,x)$ with $d(w,y)\le \varsigma$ by $\varsigma$-slimness of the triangle. But then, from $d(u,v)+d(v,w)=d(u,w)\le d(u,y)+d(y,w)\le d(u,y)+\varsigma$ we get $d(u,y)\ge d(u,v)+d(v,w)-\varsigma$. Additionally, from $d(x,w)\le d(x,y)+d(y,w)\le d(x,y)+\varsigma$ we get $d(x,y)\ge d(x,w)-\varsigma$. Hence, $d(u,x)=d(u,y)+d(x,y)\ge d(u,v)+d(v,w)-\varsigma+d(x,w)-\varsigma= d(u,v)+d(v,w)+d(w,x)-2\varsigma.$ 
\end{proof}

Note that  the results of Proposition \ref{prop:bow-metric} and Proposition \ref{prop:bow-metric-sl} are rather sharp. 
$(\delta\times \delta)$-Rectilinear grid has hyperbolicity $\delta$, slimness $\delta$, and satisfies $(\delta-1,2\delta)$-bow metric but neither $(\delta-1,2\delta-1)$-bow metric nor $(\delta,2\delta)$-bow metric. 
\medskip

From Proposition \ref{prop:bow-metric} and known in literature results on the hyperbolicity of corresponding graph classes (see Table \ref{table:hb_sl_bow}) we get also that graphs with tree-length $\lambda$, $k$-chordal graphs for every $k\ge 3$, and AT-free graphs all satisfy a ($\lambda,\mu$)-bow metric for some  values of $\lambda$ and $\mu$. Below, we give direct proofs for  those results.

\medskip

Recall that an independent set of three vertices such that each pair is joined by a path that avoids the neighborhood of the third is called an \emph{asteroidal triple}. A graph $G$ is called an \emph{AT-free graph} if it does not contain any asteroidal triple. The class of AT-free graphs contains many intersection families of graphs, including permutation graphs, trapezoid graphs, co-comparability graphs and interval graphs. 

\begin{proposition}  \label{prop:bow-metric-AT}
Every AT-free graph satisfies $(1,2)$-bow metric. 
\end{proposition}

\begin{proof} Consider arbitrary four vertices $u,v,w,x$ such that $v\in I(u,w), w\in I(v,x)$ and $d(v,w)>1$. We need to show that $d(x,u)\ge d(u,v)+d(v,w)+d(w,x)-2.$ Without loss of generality, we can choose $d(v,w)$ to be equal to $2$. Consider a common neighbor $y$ of $v$ and $w$. 
We can also assume without loss of generality $u\not= v$ and $w\not= x$. Let $v'$ be a neighbor of $v$ one step closer to $u$ and $w'$ be a neighbor of $w$ one step closer to $x$. 
Fix shortest paths $P(u,v')$, $P(w',x)$ and $P(u,x)$ connecting corresponding vertices. We know from distance requirements that $v'$ cannot be adjacent to $y,w,w'$, that $w'$ cannot be adjacent to $y,v,v'$, and that $y$ cannot have any neighbors in paths $P(u,v')$ and  $P(w',x)$. To avoid an asteroidal triple formed by vertices $v',y,w'$,  either $y \in P(u,x)$, or vertex $y$ must have a neighbor in path $P(u,x)$.  In the former case, we are done since $d(u,x) = d(u,y)+d(y,x) = d(u,v)+d(v,w)+d(w,x)$. Now in the latter case, let $z$ be any neighbor of $y$ on $P(u,x)$.  Since $d(u,y)=d(u,v)+1$ and $d(x,y)=d(x,w)+1$, we get $d(u,z)\ge d(u,v)$ and $d(x,z)\ge d(x,w)$. Therefore, $d(x,u)=d(u,z)+d(z,x)\ge d(u,v)+d(w,x)=d(u,v)+d(v,w)+d(w,x)-2.$
\end{proof}

Note that the result of Proposition  \ref{prop:bow-metric-AT} is sharp. $(1\times n)$-Rectilinear grid (so-called {\em ladder}) gives an example of an AT-free graph that satisfies a ($0,\mu$)-bow metric only for  $\mu=\Omega(n)$. Furthermore, $(1\times 2)$-rectilinear grid (so-called {\em domino}) gives an example of an AT-free graph that satisfies ($1,2$)-bow metric but not ($1,1$)-bow metric. 
\medskip

Recall that a {\em tree-decomposition} of a graph $G=(V,E)$ is a tree $T$
whose vertices $V(T)$, called {\em bags}, are subsets of $V$ such that \medskip

\noindent
(i) $\cup_{X\in V(T)}X = V$; \\
(ii) for all $uv \in E$, there exists
$X \in V(T)$ with $u, v \in X$; \\
(iii) for all $X, Y, Z \in V(T)$, if $Y$
is on the path from $X$ to $Z$ in $T$ then $X \cap Z \subseteq Y$. 
\medskip

\noindent
The {\em length} of a tree-decomposition of $G$ is the maximal diameter in $G$ of a bag of the decomposition and the {\em tree-length} of $G$ \cite{DoGa2007} is the minimum, over all tree-decompositions $T$ of $G$, of the
length of $T$.  
 
\begin{proposition}  \label{prop:bow-metric-tl}
Every graph with tree-length $\lambda$ satisfies $(\lambda,2\lambda)$-bow metric. 
\end{proposition}

\begin{proof} Consider arbitrary four vertices $u,y,w,x$ such that $y\in I(u,w), w\in I(y,x)$ and $d(y,w)>\lambda$. We need to show that $d(x,u)\ge d(u,y)+d(y,w)+d(w,x)-2\lambda.$ Consider three distance sums $A=d(u,w)+d(y,x)$, $B=d(u,x)+d(y,w)$, and $C=d(u,y)+d(w,x)$. Clearly, $A-C=2d(y,w)>2\lambda$. Furthermore,  $B\le A$ since $B=d(u,x)+d(y,w)\le d(u,y)+d(y,w)+d(w,x)+d(y,w)=A$.  

Consider now a decomposition tree $T$ of $G$ with bags
of diameter at most $\lambda$. Let $X, W, Y, U$ be
some bags of $T$ containing the vertices $x, w, y, u$, respectively.
Root the tree $T$ at $W$ and let $M$ be the lowest common
ancestor in $T$ of $X, Y$, and $U$. Let $\ell_a$ be the distance in $G$ between $a\in \{x,w,y,u\}$ and a closest to $a$ vertex in $M$.  Set $L:=\ell_x+\ell_w+\ell_y+\ell_u$.  For every $a,b\in \{x,w,y,u\}$, by the triangle inequality, we have 
$d(a,b)\le \ell_a+\ell_b+\lambda$. Hence, each of the sums $A$, $B$, and $C$ are at most $\ell_x+\ell_w+\ell_y+\ell_u+2\lambda=L+2\lambda$. 
By the choice of $M$, there exists a bag from $\{X, Y, U\}$ such that the two paths of $T$ connecting it to the two other bags must pass via $M$. 
By definition of the tree
$T$, the bag $M$ is a separator in $G$ for any pair of vertices $a, b$ from the set $\{x, y, u, w\}$ except possibly one pair. This
shows that $d(a, b)\ge \ell_a + \ell_b$ (except maybe for the unique pair $a,b$ for which $M$ is not a separator). Therefore, two largest
distance sums among $A,B$ and $C$ are larger than or equal to $L=\ell_x+\ell_w+\ell_y+\ell_u$. 

Thus, the difference between the two largest distance sums is at most $2\lambda$. We know that $A\ge B$ and $A-C>2\lambda$. Hence, $A$ is  largest among the three distance sums and $C$ cannot be second largest (as $A-C> 2\lambda$). Necessarily, $B$ is second largest and, therefore, $2\lambda\ge A-B= d(u,w)+d(y,x) -d(u,x)-d(y,w)$, i.e.,  $d(u,x)\ge d(u,y)+d(y,w)+d(w,x)-2\lambda.$
\end{proof}

\medskip 

\noindent
A graph $G$ is {\em $k$-chordal} if every induced cycle of $G$ has length at most $k$. The parameter $k$ is usually called \emph{the chordality} of $G$.
When $k=3$, $G$ is called a \emph{chordal graph}. As we mentioned earlier, each chordal graph satisfies (0,1)-bow metric. Next proposition covers all remaining cases  with $k\ge 4$.  
 
\begin{proposition}  \label{prop:bow-metric-k-ch}
Every $k$-chordal graph $G$ $(k\ge 4)$ satisfies $(\lfloor\frac{k}{4}\rfloor, \lfloor\frac{k}{2}\rfloor)$-bow metric. 
\end{proposition}

\begin{proof} Consider arbitrary four vertices $u,v,w,x$ such that $v\in I(u,w), w\in I(v,x)$ and $d(v,w)>k/4$. We need to show that $d(x,u)\ge d(u,v)+d(v,w)+d(w,x)-k/2$ holds. Since the distances are integers, the required statement will follow.  Consider arbitrary shortest paths $P(u,v)$, $P(v,w)$, $P(w,x)$ and $P(u,x)$ connecting the corresponding vertices. First we prove two claims. 

\begin{claim}\label{cl:one}
    For every vertices $a\in P(u,v)$ and $b\in P(w,x)$, $d(a,b)\ge d(v,w)>k/4$ holds. 
\end{claim}
We have $d(v,w)+d(w,x)=d(v,x)\le d(x,b)+d(a,b)+d(a,v)$ and 
$d(w,v)+d(v,u)=d(w,u)\le d(u,a)+d(a,b)+d(b,w)$. Summing up these inequalities, we get $2d(v,w)+d(u,v)+d(w,x)\le d(x,b)+d(a,b)+d(a,v)+d(u,a)+d(a,b)+d(b,w)=d(x,w)+2d(a,b)+d(u,v)$, i.e., $d(a,b)\ge d(v,w)>k/4$.   {\hfill $\diamond$}

\begin{claim}\label{cl:two}
 For every vertices $a\in P(w,v)$ and $b\in P(u,x)$,  $d(u,x)\ge d(u,v)+d(v,w)+d(w,x)-2d(a,b)$ holds. 
\end{claim}
We have $d(v,w)+d(w,x)=d(v,x)\le d(x,b)+d(a,b)+d(a,v)$ and 
$d(w,v)+d(v,u)=d(w,u)\le d(u,b)+d(a,b)+d(a,w)$. Summing up these inequalities, we get $2d(v,w)+d(u,v)+d(w,x)\le d(x,b)+d(a,b)+d(a,v)+d(u,b)+d(a,b)+d(a,w)=d(x,u)+2d(a,b)+d(w,v)$, i.e., $d(u,x)\ge d(u,v)+d(v,w)+d(w,x)-2d(a,b)$.   {\hfill $\diamond$}
\medskip

In Claim \ref{cl:two}, $d(a,b)\le k/4$ implies $d(u,x)\ge d(u,v)+d(v,w)+d(w,x)-k/2$.  So, we may assume, in what follows,  that $d(a,b)> k/4$ holds for every $a\in P(w,v)$ and $b\in P(u,x)$. 
In particular, we have $d(u,v)> k/4$ and  $d(w,x)> k/4$.  We prove below that this situation leads to a contradiction with $G$ being a $k$-chordal graph.  

Consider a cycle $C$ formed by paths $P(u,v)$, $P(v,w)$, $P(w,x)$ and $P(x,u)$. Without loss of generality, we may assume that path  $P(u,x)$ shares with paths $P(u,v)$ and $P(x,w)$ only end vertices $u$ and $x$, respectively, i.e., this cycle $C$ is simple (see also Claim \ref{cl:one} and Claim \ref{cl:two}).  Since $P(u,v)\cup  P(v,w)$ is a shortest path between $u$ and $w$ of length greater that $k/2$, this simple cycle $C$ has length greater that $k$. By  $k$-chordality of $G$, $C$ must have chords. From the discussion above, any chord of $C$ must connect a vertex of $P(u,v)$ with a vertex of $P(u,x)$ or a vertex of $P(w,x)$ with a vertex of $P(u,x)$. Chose a chord $ab$ between  $P(u,v)$ and  $P(u,x)$ with $a\in P(u,v)$ and  $b\in P(u,x)$ 
(if it exists) with the maximum sum $d(a,u)+d(b,u)$ and a 
chord $st$ between  $P(u,x)$ and  $P(x,w)$ with $s\in P(u,x)$ and  $t\in P(x,w)$ 
(if it exists) with the maximum sum $d(s,x)+d(t,x)$ (in case chord $ab$ exists, we choose $st$ such that $d(b,x)\ge d(s,x)$ and $d(s,x)+d(t,x)$ is maximal among all such choices of $st$; note that such $st$ may not exist).  Now consider a subcycle $C'$ of $C$ 
formed by path $P(v,w)$, subpath $P(w,t)$ of $P(w,x)$, edge $ts$ (it such a chord $ts$ of $C$ does not exist then let $s=t=x$), subpath $P(s,b)$ of $P(x,u)$, edge $ba$ (it such a chord $ba$ of $C$ does not exist then let $a=b=u$), and subpath $P(a,v)$ of $P(u,v)$. Since $P(a,v)\cup  P(v,w)$ is a shortest path between $a$ and $w$ of length greater that $k/2$ (recall that $d(v,w)>k/4$ and $d(v,a)\ge k/4$ as a consequence of $d(v,b)>k/4$), this simple cycle $C'$ has also length greater that $k$. By  $k$-chordality of $G$, $C'$ must have chords. Again, from the discussion above, any chord of $C'$ must connect a vertex of $P(a,v)$ with a vertex of $P(b,s)$ or a vertex of $P(w,t)$ with a vertex of $P(b,s)$. However, such chords cannot exist by the choices of chords $ab$ and $st$. The contradiction obtained completes the  proof.  
\end{proof}
 

\section{When ($\lambda,\mu$)-bow metric guaranties hyperbolicity}\label{Sec:imply-hyp}

We  are now looking for some converse relationship
between hyperbolicity and ($\lambda,\mu$)-bow metric.  
We will need the following important result from \cite{delta-hyp-1st}.

\begin{proposition} [\cite{delta-hyp-1st}]  \label{prop:delta-hyp-1st} 
If the intervals of a graph $G$
are $p$-thin and the metric triangles of $G$ have sides of length
at most $q$, then $G$ has slimness at most 
$2p+\frac{q}{2}$ and, hence, $G$ is $(4p + q+\frac{1}{2})$-hyperbolic. 
\end{proposition}

We complement Proposition  \ref{prop:delta-hyp-1st} with the following result which generalizes a result from \cite{obstructions} on Helly graphs.

\begin{proposition} \label{prop:delta-hyp-2nd} 
If the intervals of a graph $G$
are $p$-thin and the metric triangles of $G$ have sides of length
at most $q$, then $G$ is $(2q+\frac{p}{2})$-hyperbolic  and, hence, $G$ has slimness at most $6q+\frac{3}{2}p+\frac{1}{2}$.  
\end{proposition}
\begin{proof} 
Consider four vertices $u,v,w,x$ of $G$ and the three distance sums $A:=d(u,w)+d(v,x)$, $B:=d(v,w)+d(u,x)$ and $C:=d(u,v)+d(w,x)$. Without loss of generality, let $A$ be largest out those three distance sums. 
Assume also that the quadruple $u,v,w,x$ realizes the hyperbolicity $\delta$ of $G$, i.e., $2\delta=A-\max\{B,C\}$. 

Let $v'w'x'$ be a quasi-median of $v,w,x$ and $v''u''x''$ be a quasi-median of $v,u,x$. Since $v'w'x'$ and $v''u''x''$ are metric triangles of $G$, we have $d(x',v')\le q$ and $d(x'',v'')\le q$. We also have (see Fig.~\ref{fig:zero-ab} for an illustration)
\begin{itemize}
    \item[]  $d(v,w)=d(v,v')+d(v',w')+d(w',w),$
    \item[]  $d(x,w) = d(x,x') + d(x',w') + d(w',w),$
    \item[]  $d(v,x) = d(v,v') + d(v',x') + d(x',x),$
\end{itemize}
and 
\begin{itemize}
    \item[]  $d(v,u)=d(v,v'')+d(v'',u'')+d(u'',u),$
    \item[]  $d(x,u) = d(x,x'') + d(x'',u'') + d(u'',u),$
    \item[]  $d(v,x) = d(v,v'') + d(v'',x'') + d(x'',x).$
\end{itemize}
Without loss of generality, we may assume $d(v,v')\ge d(v,v'')$. 

  \begin{figure}[htb]
    \begin{center} 
      \begin{minipage}[b]{16cm}
        \begin{center} 
          \includegraphics[height=16cm]{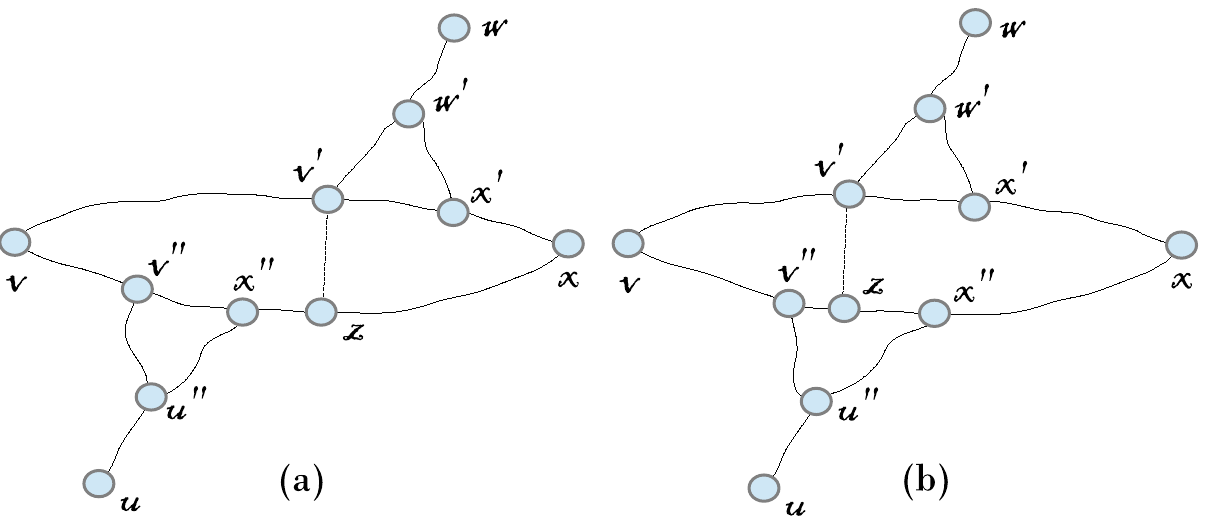}
        \end{center} \vspace*{-106mm}
        \caption{\label{fig:zero-ab} An illustration to the proof of Proposition \ref{prop:delta-hyp-2nd} .} %
      \end{minipage}
    \end{center}
   \vspace*{-1mm}
  \end{figure}
 \medskip

\noindent 
{\em Case   $d(v,v')\ge d(v,x'')$}.
\medskip

\noindent 
Consider a vertex $z\in I(x'',x)$ with $d(x,v')= d(x,z)$. See Fig.~\ref{fig:zero-ab}(a) for an illustration. Since $z,v'\in I(v,x)$ and $d(x,v')= d(x,z)$, we have $d(z,v')\le p$.  

First we show that $B\ge C$ holds. Assume, by way of contradiction, that $C>B$. 
Then, 
\begin{align*}
0> B-C   & =d(v,w)+d(u,x)-d(u,v)-d(w,x) \\
           & =(d(v,v')+d(v',w')+d(w',w))+(d(u,u'')+d(u'',x'')+d(x'',x)) \\
           & ~~~~~ -((d(v,v'')+d(v'',u'')+d(u'',u))-(d(x,x')+d(x',w')+d(w',w)) \\
           & =d(v,v')-d(v,v'') +d(v',w')-d(x',w') + d(u'',x'')-d(v'',u'') +d(x'',x)-d(x,x').
\end{align*}
Since $d(x',w')\le d(v',w')+d(v',x')$, $d(v'',u'')\le d(x'',v'')+d(x'',u'')$, $d(v,v')-d(v,v'')=d(v'',x'')+d(x'',z)$ and $d(x'',x)-d(x,x')=d(x'',z)+d(v',x')$, we get 
\begin{align*}
0 & > (d(v,v')-d(v,v'')) +(d(v',w')-d(x',w')) + (d(u'',x'')-d(v'',u'')) +(d(x'',x)-d(x,x')) \\ 
& \ge (d(v'',x'')+d(x'',z)) -d(x',v') -d(v'',x'')+(d(x'',z)+d(v',x')) \\
& = 2d(x'',z),  
\end{align*}
which is impossible. 

A contradiction obtained shows that $B=\max\{B,C\}$. Hence,    
\begin{align*}
   2\delta= A-B & =d(u,w)+d(v,x)- d(v,w)-d(u,x)   \\
                & \le (d(u,x'')+d(x'',z)+d(z,v')+d(v',w))+d(v,x)   \\
     & ~~~~~~~~~~~~~~~~~~~~~~~~~-(d(v,v')+d(v',w))-(d(u,x'')+d(x'',x)) \\
     &\le d(x'',z)+p+d(v,x)-d(v,v')-d(x'',x) \\
     &= p,   
     \end{align*}
since $d(x'',z)+d(z,x)=d(x'',x)$, $d(v,x)=d(v,v')+d(v',x)$ and $d(z,x)=d(v',x)$. 
\medskip

\noindent 
{\em Case   $d(v,v'')\le d(v,v')< d(v,x'')$}.
\medskip

\noindent 
Consider a vertex $z\in I(v'',x'')$ with $d(x,v')= d(x,z)$. See Fig.~\ref{fig:zero-ab}(b) for an illustration. Since $z,v'\in I(v,x)$ and $d(x,v')= d(x,z)$, we have $d(z,v')\le p$. As in the previous case, we have 
\begin{align*}
B-C   & =d(v,w)+d(u,x)-d(u,v)-d(w,x) \\
           & =(d(v,v')+d(v',w')+d(w',w))+(d(u,u'')+d(u'',x'')+d(x'',x)) \\
           & ~~~~~ -(d(v,v'')+d(v'',u'')+d(u'',u))-(d(x,x')+d(x',w')+d(w',w)) \\
           & =d(v,v')-d(v,v'') +d(v',w')-d(x',w') + d(u'',x'')-d(v'',u'') +d(x'',x)-d(x,x') \\
           & \ge d(v,v')-d(v,v'')-d(x',v') -d(v'',x'') +d(x'',x)-d(x,x') \\
           & =d(v'',z)-d(x',v') -d(v'',x'') +d(x'',x)-d(x,x')
\end{align*} 
since $d(v,v')-d(v,v'')=d(v'',z)$. 

Now, if $C>B$, then we get $2\delta=A-C$ and from above 
\begin{align}
d(x',v')+d(v'',x'')-d(v'',z) >d(x'',x)-d(x,x').
\end{align} 
Consequently, 
\begin{align*}
   2\delta= A-C & =d(v,x)+d(u,w)- d(v,u)-d(w,x)   \\
                & \le d(v,x)+(d(u,v'')+d(v'',z)+d(z,v')+d(v',x')+d(x',w))   \\
     & ~~~~~~~~~~~~~~~~~~~~~~~~~-(d(v,v'')+d(v'',u))-(d(w,x')+d(x',x)) \\
     & \le d(v'',x'')+d(x'',x)+d(v'',z)+p+d(v',x')-d(x,x'),
\end{align*}
since   $d(z,v')\le p $ and $d(v,x)-d(v,v'')=d(v'',x'')+d(x'',x)$. Furthermore, from inequality (1), we get  
\begin{align*}
   2\delta  &< d(v'',x'')+d(v'',z)+p+d(v',x')+d(x',v')+d(v'',x'')-d(v'',z) \\
     & = 2(d(x'',v'')+d(v',x'))+p\\
     & \le 4q+p.     
     \end{align*}

When $B\ge C$ we have 
\begin{align*}
   2\delta= A-B & =d(u,w)+d(v,x)- d(v,w)-d(u,x)   \\
                & \le (d(u,x'')+d(x'',z)+d(z,v')+d(v',w))+d(v,x)   \\
     & ~~~~~~~~~~~~~~~~~~~~~~~~~-(d(v,v')+d(v',w))-(d(u,x'')+d(x'',x)) \\
     &\le d(x'',z)+p+d(v,x)-d(v,v')-d(x'',x) \\
     & = d(x'',z)+p+d(v',x)-d(x'',x) \\
     & = 2d(x'',z)+p \\ 
     &\le 2q + p.    
     \end{align*}

In all cases we got $2\delta< 4q+p$ or $2\delta\le 2q+p$, i.e.,  $\delta\le 2q+p/2$. The bound on slimness follows from Proposition \ref{prop:sl-vs-hb}. 
\end{proof} 

The proof of Proposition \ref{prop:delta-hyp-2nd} allows to get better bounds for the graphs with equilateral metric triangles. 

\begin{corollary}\label{cor:delta-hyp-2nd-equidist} 
If the intervals of a graph $G$
are $p$-thin and the metric triangles of $G$ are equilateral of size at most  $q$, then $G$ is $\frac{p + q}{2}$-hyperbolic  and, hence, $G$ has slimness at most $\frac{3(p+q)+1}{2}$. 
\end{corollary}

\begin{proof} One can follow the proof of Proposition \ref{prop:delta-hyp-2nd} and even simplify it because in this case we have equalities $d(u,x'')=d(u,v'')$ and $d(w,x')=d(w,v')$. 

According to the more general proof of Proposition \ref{prop:delta-hyp-2nd}, we need to consider only the case when  $d(v,v'')\le d(v,v')< d(v,x'')$ (see Fig.~\ref{fig:zero-ab}(b)). 
Since $d(u,x'')=d(u,v'')$ and $d(w,x')=d(w,v')$, we have 
\begin{align*}
B-C   & =d(v,w)+d(u,x)-d(u,v)-d(w,x) \\
           & =(d(v,v')+d(v',w))+(d(u,x'')+d(x'',x))  \\
           & ~~~~~ -(d(v,v'')+d(v'',u))-(d(x,x')+d(x',w)) \\
           & =d(v,v')-d(v,v'')  +d(x'',x)-d(x,x') \\
           & =d(v'',z)+d(x'',x)-d(x,x'). 
\end{align*} 

Now, if $C>B$, then we get $2\delta=A-C$ and from above 
\begin{align}
d(v'',z) <d(x,x')-d(x'',x).
\end{align} 
Consequently, 
\begin{align*}
   2\delta= A-C & =d(v,x)+d(u,w)- d(v,u)-d(w,x)   \\
                & \le d(v,x)+(d(u,v'')+d(v'',z)+d(z,v')+d(v',w))   \\
     & ~~~~~~~~~~~~~~~~~~~~~~~~~-(d(v,v'')+d(v'',u))-(d(w,x')+d(x',x)) \\
     & \le d(v'',x)+d(v'',z)+p-d(x,x'),
\end{align*}
since $d(w,x')=d(w,v')$, $d(z,v')\le p $ and $d(v,x)-d(v,v'')=d(v'',x)$. Furthermore, from inequality (2), we get  
\begin{align*}
   2\delta  &< d(v'',x)+d(x,x')-d(x'',x)+p-d(x,x') \\
     & = d(x'',v'')+p\\
     & \le q+p.     
     \end{align*}

When $B\ge C$ we have (recall that $d(u,x'')=d(u,v'')$ and $d(v',x)=d(z,x)$)  
\begin{align*}
   2\delta= A-B & =d(u,w)+d(v,x)- d(v,w)-d(u,x)   \\
                & \le (d(u,v'')+d(v'',z)+d(z,v')+d(v',w))+d(v,x)   \\
     & ~~~~~~~~~~~~~~~~~~~~~~~~~-(d(v,v')+d(v',w))-(d(u,x'')+d(x'',x)) \\
     &\le d(v'',z)+p+d(v,x)-d(v,v')-d(x'',x) \\
     & = d(v'',z)+p+d(v',x)-d(x'',x) \\
     & = d(v'',z)+p+d(z,x)-d(x'',x) \\
     & = d(v'',z)+d(z,x'')+p \\ 
     & = d(x'',v'')+p \\ 
     &\le q + p.    
     \end{align*}

 The bound on slimness follows from Proposition \ref{prop:sl-vs-hb}.
\end{proof}

Now we turn to the graphs satisfying a  ($\lambda,\mu$)-bow metric. 

\subsection{Graphs in which side lengths of metric triangles are bounded or interval thinness is bounded. }
First we consider the graphs in which side lengths of metric triangles are bounded.

\begin{proposition} \label{prop:bow-int-thin} Let $G$ be a graph satisfying a  ($\lambda,\mu$)-bow metric. If the metric triangles of $G$ have sides of length at most $q$, then the intervals of $G$ are $p$-thin for $p\le \max\{\mu, q+2\lambda\}.$ 
\end{proposition}

\begin{proof} 
Consider arbitrary vertices $u,v$ in $G$ and arbitrary vertices $x,y\in S_k(u,v)$, where $1\le k \le d(u,v)-1$. Assume $d(x,y)> \mu$. Let $x'y'v'$ be a quasi-median of $x,y,v$. Since $x'y'v'$ is a metric triangle of $G$, we have $d(x',y')\le q$. If $d(x,x')>\lambda$, then we can apply ($\lambda,\mu$)-bow metric to $x'\in I(x,y)$ (this follows from the definition of a quasi-median; $x'$ is on a shortest path from $y$ to $x$ passing through $y'$) and $x\in I(x',u)$ (this inclusion follows from the definition of a quasi-median and of an interval; $x$ is on a shortest path from $v$ to $u$ passing through $v'$ and $x'$; see Fig. \ref{fig:one})
and get $d(u,y)\ge d(u,x)+d(x,y)-\mu >d(u,x)$ (since $d(x,y)>\mu$). As both $x$ and $y$ are at distance $k$ from $u$, we arrived at a contradiction.  Thus, $d(x,x')\le\lambda$ must hold. By symmetry,  $d(y,y')\le\lambda$ holds, too. Consequently, $d(x,y)\le d(x,x')+d(x',y')+d(y',y)\le \lambda+q+\lambda= q+2\lambda$. 

We assumed $d(x,y)> \mu$ and got $d(x,y)\le q+2\lambda$. Hence, $d(x,y)\le \max\{\mu, q+2\lambda\}.$
\end{proof}

  \begin{figure}[htb]
    \begin{center} \vspace*{-9mm}
      \begin{minipage}[b]{15cm}
        \begin{center} 
          \hspace*{1mm}
          \includegraphics[height=16cm]{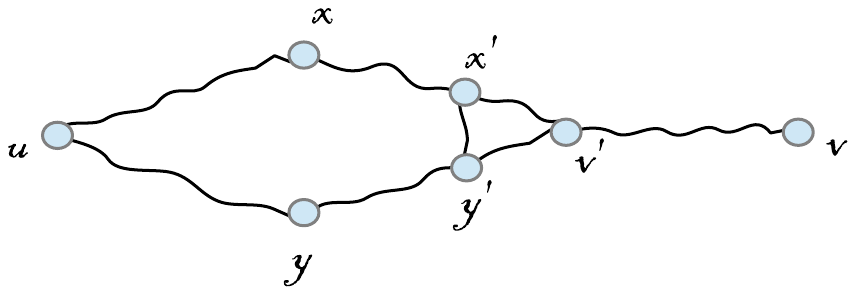}
        \end{center} \vspace*{-126mm}
        \caption{\label{fig:one} An illustration to the proof of Proposition \ref{prop:bow-int-thin}.} %
      \end{minipage}
    \end{center}
   \vspace*{-5mm}
  \end{figure}

From Proposition \ref{prop:delta-hyp-1st}, Proposition \ref{prop:delta-hyp-2nd},  and Proposition \ref{prop:bow-int-thin}, we conclude. 

\begin{theorem} \label{cor:bow-int-hyperb} Let $G$ be a graph satisfying a  ($\lambda,\mu$)-bow metric. If the metric triangles of $G$ have sides of length at most $q$, then the following inequalities hold for the hyperbolicity $\delta$ and the slimness $\varsigma$ of $G$:
\begin{align*}
   \delta\le\min & \{\frac{1}{2}\max\{\mu,q+2\lambda\}+2q, 
                      ~~4\max\{\mu,q+2\lambda\} + q+\frac{1}{2}\} \\
    \varsigma\le\min & \{\frac{3}{2}\max\{\mu,q+2\lambda\}+6q+\frac{1}{2}, 
                         ~~2\max\{\mu,q+2\lambda\}+\frac{q}{2} \}.       
     \end{align*}
\end{theorem}

Note that $\alpha_2$-metric graphs give an example of graph classes where the metric triangles have unbounded side length, yet all $\alpha_2$-metric graphs are (9/2)-hyperbolic (see \cite{alpha-hyperb}).  The intervals of $\alpha_i$-metric graphs are $(i+1)$-thin \cite{alpha-hyperb}. 

A set $S\subseteq V$ of a graph $G=(V,E)$ is called {\em $\lambda$-set} if $d_G(x,y)\le \lambda$ holds for every $x,y\in S$. Clearly, a $0$-set is a singleton and a $1$-set is a complete subgraph  (a clique) of $G$. A graph $G=(V,E)$ is called a {\em $\lambda$-generalized Helly graph} if for any subset $M\subseteq V$ and any $r$-function $r: M\rightarrow R^+\cup\{0\}$, inequalities $d(x,y)\leq r(x)+r(y)+\lambda$, for every $x,y\in M$, imply that there is a $\lambda$-set $S$ in $G$ such that $d(v,S)\leq r(v)$ for all $v\in M$~\cite{GenHelly}. The set $S$ is called an $r$-dominating set of $M$. Equivalently, one says that $S$ $r$-dominates $M$.  
The $0$-generalized Helly graphs are exactly the Helly graphs, i.e., the graphs whose balls satisfy the {\em Helly property} (that
is, every collection of pairwise intersecting balls has a nonempty common intersection).  
It is known \cite{GenHelly} that the length of the longest side of any metric triangle of a $k$-generalized Helly graph is at most $\max\{k+1,2k\}$. Hence, for $k>0$, we have the following Corollary \ref{cor:bow-gH-hyperb}. For the case of Helly graphs, i.e, when $k=0$, see Section \ref{subsec:MedianHellyETC}.

 \begin{corollary} \label{cor:bow-gH-hyperb} Let $G$ be a $k$-generalized Helly graph $(k>0)$ satisfying a  ($\lambda,\mu$)-bow metric.  
 Then the following inequalities are true for the hyperbolicity $\delta$ and the slimness $\varsigma$ of $G$:
\begin{align*}
   \delta\le\min & \{\frac{1}{2}\max\{\mu,2(k+\lambda)\}+4k, 
                      ~~4\max\{\mu,2(k+\lambda)\} + 2k+\frac{1}{2}\} \\
    \varsigma\le\min & \{\frac{3}{2}\max\{\mu,2(k+\lambda)\}+12k+\frac{1}{2}, 
                         ~~2\max\{\mu,2(k+\lambda)\}+k\}.       
     \end{align*}
\end{corollary}

($\lambda,\mu$)-Bow metrics naturally generalize $\alpha_i$-metrics. From \cite{alpha-centers,alpha-hyperb} we know that, although any $\alpha_i$-metric graph has the interval thinness at most $i+1$, there are even $\alpha_1$-metric graphs whose $1$-subdivision graphs satisfy $\alpha_{i}$-metrics only for $i=\Omega(n)$. However, this generalization is more robust. The $1$-subdivision of a ($\lambda,\mu$)-bow metric graph must satisfy ($2\lambda+2,2\mu+2$)-bow metric. Recall that the $1$-subdivision graph $\Sigma(G)$ of a graph $G$ is obtained by replacing all edges $e=uv$ of $G$ by internally vertex-disjoint paths $[u,e,v]$ of length two. 

\begin{lemma} \label{lm:1-subdiv}
    If a graph $G$ satisfies $(\lambda,\mu)$-bow metric, then its $1$-subdivision $H$ satisfies ($2\lambda+2,2\mu+2$)-bow metric. 
\end{lemma}
\begin{proof}
    Let $x',u',v',y'$ ($x'\neq u'$, $y'\neq v'$), be vertices of $H$ such that $u' \in I_H(x',v')$, $v' \in I_H(u',y')$ and $d_H(u',v') > 2\lambda+2$. Let $x$ (resp., $y$) be a  vertex of $G$ closest to $x'$ (resp., $y'$) on some shortest path of $H$ connecting $x'$ (resp., $y'$) with $u'$ (resp., $v'$). 
    We pick also two vertices $u,v$ of $G$ that are at maximal distance on some shortest path $P_H(u',v')$.
    Note that $d_H(u,u'),d_H(v,v'),d_H(x,x'),d_H(y,y') \le 1$.
    In particular, $d_H(u,v) \ge d_H(u',v') -2 > 2\lambda$, and therefore $d_G(u,v) > \lambda$.
    Since $G$ satisfies $(\lambda,\mu)$-bow metric, $d_G(x,y) \ge d_G(x,u) + d_G(u,y) - \mu$.
    In particular, $d_H(x,y) \ge d_H(x,u) + d_H(u,y) - 2\mu$.
    Since $u$ is on a shortest path $P_H(u',v')$, $u' \in I_H(x,u)$ and $u \in I_H(u',y)$, which implies $d_H(x,u) + d_H(u,y) = d_H(x,u') + d_H(u',y)$.
    Hence, $d_H(x,y) \ge d_H(x,u') + d_H(u',y) - 2\mu$. Since, $d_H(x,y)\le d_H(x',y')+2$, we get $d_H(x',y') \ge d_H(x,u') + d_H(u',y) - 2\mu-2$.
\end{proof}



Thus, we can state the following. 
\begin{theorem} \label{cor2:bow-int-hyperb}  If in every graph satisfying ($\lambda,\mu$)-bow metric  ($\lambda\ge 0, \mu\ge 0$) the interval thinness is  bounded by $f(\lambda,\mu)$, then the hyperbolicity of a graph satisfying a ($\lambda,\mu$)-bow metric is at most doubly exponential in $O(f(\lambda,\mu))$. 
\end{theorem}

\begin{proof}
  Papasoglu~\cite{Pap95} proved that the hyperbolicity of $G$ is at most doubly exponential in the interval thinness of $\Sigma(G)$. 
\end{proof}

Theorem \ref{cor:bow-int-hyperb}  and Theorem \ref{cor2:bow-int-hyperb} suggest the following interesting remark. 

\begin{remark}
    If there exists a counterexample graph 
    to our conjecture, it must satisfy ($\lambda,\mu$)-bow metric but have unbounded side lengths of metric triangles  and an unbounded interval thinness. 
\end{remark}

\subsection{Meshed Graphs}\label{sec:meshed-gr} 
By establishing a new characterization of meshed graphs, we prove here that our conjecture is true for all meshed graphs. Meshed graphs  with a ($\lambda,\mu$)-bow metric are $\delta$-hyperbolic for some $\delta$ linearly depending only on $\lambda$ and $\mu$. The hyperbolicity constant can be further sharpened for Helly graphs, median graphs, $k$-generalized Helly graphs,  graphs with convex balls, and others. 

A graph $G$ is called {\em meshed} \cite{meshed} (see also \cite{BaCh-survey}) if for any three vertices $u, v, w$ with $d(v, w) = 2$, there exists a common neighbor $x$ of $v$ and $w$ such that $2d(u, x) \le d(u, v) + d(u, w)$. Meshed graphs are thus characterized by some (weak) convexity
property of the distance functions $d(\cdot, u)$ for $u\in V$ (see \cite{meshed,BaCh-survey}). This condition ensures that all
balls centered at complete subgraphs of a meshed graph $G$ induce isometric subgraphs and 
that every cycle can be written as a modulo 2 sum of cycles of lengths 3 and 4 (see \cite{BaCh-survey}). 


A graph $G$ is weakly modular~\cite{BaCh1996,BaMu1991,Ch1989-m-tr}
if its distance function $d(\cdot,\cdot)$ satisfies the
following conditions: \vspace{1mm} \\
{\em triangle condition}: for any three vertices $u, v, w$ with $1 = d(v, w) < d(u, v) =d(u, w)$ there exists a common neighbor $x$ of $v$ and $w$ such that $d(u, x) = d(u, v)-1$;\vspace{1mm} \\
{\em quadrangle condition}: for any four vertices $u, v, w, z$ with $d(v, z) = d(w, z) = 1$ and
$2 = d(v, w) \le d(u, v) = d(u, w) = d(u, z) -1$, there exists a common neighbor $x$ of
$v$ and $w$ such that $d(u, x) = d(u, v) -1$. \medskip

Modular graphs \cite{modular}, pseudo-modular graphs
\cite{ps-modular}, pre-median graphs \cite{Chastand}, weakly median graphs \cite{weakly-median}, quasi-median graphs \cite{ChaGraSa89,Chastand97,Wilkeit}, dual polar
graphs \cite{Cam1982-d-polar}, median graphs \cite{Bandelt1984,BanHed1983,Isbell1980,Mulder1980,van-de-Vel-book}, distance-hereditary graphs \cite{BaMu1986-dhg,Howorka}, bridged graphs \cite{bridged-FJ,bridged-SCh}, Helly graphs \cite{BanPesch,BanPris,Dr_thesis,Dr_Helly,NowRiv,Quil}, chordal graphs \cite{Ch1986-ch}, and dually chordal graphs \cite{DuChGr,Dr_HT,DrPrCh} are all instances of weakly modular  graphs. 
All weakly modular graphs are meshed~\cite{BaCh-survey}. 
Basis graphs of matroids \cite{Maurer73} and even $\Delta$-matroids \cite{delta-matr},   1-generalized Helly graphs \cite{GenHelly} as well as the graphs in which all median sets induce connected or isometric subgraphs \cite{conn-med} are meshed but in general not weakly modular. The icosahedron graph constitutes another example of a meshed graph
that is not weakly modular. 

First, for completeness, we present a simple proof of the following fact mentioned in \cite{BaCh-survey}.  
\begin{lemma} \label{lm:isom-balls} 
In a  meshed graph each ball induces an isometric subgraph.  
\end{lemma}

\begin{proof} Consider in a meshed graph $G$ a ball $B(w,k)$ and two arbitrary vertices $u,v\in B(w,k)$. Among all shortest paths connecting $u$ and $v$ in $G$ consider a shortest path $P(u,v)$ whose sum $\sigma:=\sum_{x\in P(u,v)} d(x,w)$ is smallest. Consider the closest to $v$ vertex $y$ of $P(u,v)$ such that $d(y,w)=\max\{d(x,w): x\in P(u,v)\}$ (i.e., $y$ is a  vertex of  $P(u,v)$ most distant from $w$). If $d(y,w)\le k$ then we are done - all vertices of $P(u,v)$ are in $B(w,k)$.  If $d(y,w)>k$, then we can apply the meshedness of $G$ to $v',y,u'$, where $v'$ ($u'$ respectively) is the neighbor of $y$ in  $P(u,v)$ closest to $v$ (to $u$, respectively). 
  Since $d(v',w)<d(y,w)$ and $d(u',w)\le d(y,w)$ by the choice of $y$, the meshedness of $G$ implies that there exists a vertex $y'$ adjacent to $u',v'$ and at distance less than $d(w,y)$ from $w$. 
  The latter contradicts our assumption that  $P(u,v)$ had smallest sum $\sigma$; replacing $y$ with $y'$ in $P(u,v)$ will produce a shortest path with a smaller sum. 
\end{proof}

The metric triangles of meshed graphs are equilateral \cite{conn-med}. Metric triangles of
weakly modular graphs are somewhat more special: namely, a graph $G$ is weakly
modular if and only if for every metric triangle $uvw$ all vertices of the interval
$I(u, v)$ are at the same distance $k= d(u, w)$ from $w$ \cite{Ch1989-m-tr} (such metric triangles are called {\em strongly equilateral} \cite{Ch1989-m-tr}). In the next Theorem \ref{th:metric-tr}, we generalize this to a characterization of meshed graphs by exchanging $I(u, v)$ by the existence of a shortest path $P(u, v)$ with the same condition. 


\begin{theorem} \label{th:metric-tr} 
A graph $G$ is meshed if and only if for every metric triangle $uvw$ there is a shortest path $P(u,v)$ connecting $u$ and $v$ such that 
all vertices of $P(u,v)$ are at the same distance from $w$. 
\end{theorem}

\begin{proof}
Consider a metric triangle  $uvw$ in a meshed graph $G$. Let $P(u,v):=(u=u_0,u_1,\dots,u_i,$ $\dots, u_\ell=v)$ be a shortest path connecting $u$ and $v$ in $G$ such that all vertices of $P(u,v)$ are in $B(w,k)$ where $k=\max\{d(u, w),d(v,w)\}$   (such a shortest path exists by Lemma \ref{lm:isom-balls}). 
Let, without loss of generality, $k=d(u, w)\ge d(v,w)$.
Since $I(u,v)\cap I(u,w)=\{u\}$, vertex $u_1$ cannot be closer than $u$ to $w$. That is, $d(u_1,w)=k$. For the sake of contradiction, let $u_i$ be a vertex of $P(u,v)$ with smallest index $i$ such that $d(u_i,w)=d(u_{i-1},w)-1$. Necessarily, $i\ge 2$, $d(u_i,w)=k-1$ and $d(u_{i-1},w)=d(u_{i-2},w)=k$. By the meshedness of $G$ applied to  $u_{i-2},u_{i-1},u_i$, we can find a vertex $u'_{i-1}$ adjacent to $u_{i-2}$ and $u_i$ and at distance $k-1$ from $w$. If $i\ge 3$, we can again apply the meshedness of $G$ now to   $u'_{i-1},u_{i-2},u_{i-3}$ and get a new vertex $u'_{i-2}$ adjacent to $u_{i-3}$ and $u'_{i-1}$ and at distance $k-1$ from $w$. For $i\ge 4$, we keep applying the meshedness of $G$ until (see Fig. \ref{fig:two})  
we get a shortest path $P'(u,v):=(u=u_0,u'_1,\dots,u'_{i-1},u_i,\dots, u_\ell=v)$ between $u$ and $v$ whose vertex $u'_1$ is at distance $k-1$ from $w$. That is, $u'_1\in I(u,v)\cap I(u,w)$. The latter is impossible since $I(u,v)\cap I(u,w)=\{u\}$. Thus, such a vertex $u_i$ with $d(u_i,w)=k-1$ cannot exist, i.e., all vertices of $P(u,v)$ are at distance $k$ from $w$. 

  \begin{figure}[htb]
    \begin{center} \vspace*{-12mm}
      \begin{minipage}[b]{15cm}
        \begin{center} 
          \hspace*{25mm}
          \includegraphics[height=16cm]{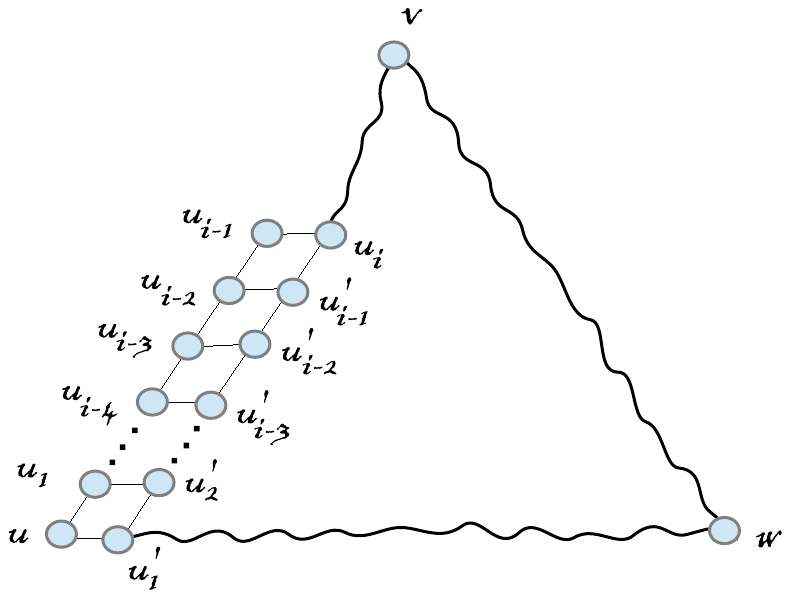}
        \end{center} \vspace*{-96mm}
        \caption{\label{fig:two} An illustration to the proof of Theorem \ref{th:metric-tr}.} %
      \end{minipage}
    \end{center}
   \vspace*{-1mm}
  \end{figure}

To prove the converse, consider arbitrary three vertices $u, v, w$ with $d(v, w) = 2$ and $d(u, v)\le d(u, w)$. We need only to consider two cases: when $d(u, v)=d(u, w)$ and when $d(u, v)=d(u, w)-1$. The reason is the following. Note that, if in $G$ for every metric triangle $uvw$ there is a shortest path $P(u,v)$ connecting $u$ and $v$ such that all vertices of $P(u,v)$ are at the same distance $k= d(u, w)$ from $w$, then in particular all metric triangles in $G$ are equilateral.  
Consider a quasi-median $u'v'w'$ of $u,v,w$. Since all metric triangles in $G$ are equilateral, we must have $d(v',u')=d(u',w')=d(v',w')=k$ where $k\in \{0,1,2\}$. This implies $d(u, v)=d(u, w)$ 
or $d(u, v)=d(u, w)-1$. 

\medskip\noindent 
{\em Case  $d(u, v)=d(u, w)$}.

\medskip\noindent 
In this case, since $d(v,w)=2$ and $d(u, v)=d(u, w)$,  necessarily, $k\in \{0,2\}$. If $k=0$ then $v'=u'=w'$ and $d(u, u')=d(u, v)-1= d(u, w)-1$, implying $2d(u, u') < d(u, v) + d(u, w).$  If $k=2$ then $v=v'$, $w=w'$ and for a metric triangle $vu'w$ there must exist a shortest path $P=(v,x,w)$ (of length two) between $v$ and $w$ such that $d(x,u')=d(v,u')=d(w,u')=2$. 
Since $d(u,u')=d(u,v)-2=d(u,w)-2$, we have $2d(u, x) \le d(u, v) + d(u, w).$

\medskip\noindent 
{\em Case  $d(u, v)=d(u, w)-1$}.

\medskip\noindent 
In this case, since $d(v,w)=2$ and $d(u, v)=d(u, w)-1$,  necessarily, $k=1$. Furthermore, we must have $v=v'$, $w'$ is adjacent to $v',w,u'$, and $u'$ is adjacent to $v'$ and at distance $d(u,v)-1$ from $u$. Consequently, $2d(u,w')\le 2d(u,v) < d(u,v)+d(u,w)$. 
   \end{proof}

Recently we learned about a new paper \cite{chepoi-new}. Our Theorem \ref{th:metric-tr} can also be deduced from its Lemma 15 and Lemma 16. 
\medskip

From Theorem \ref{th:metric-tr}  and its proof we have the following immediate corollaries. 
   
 \begin{corollary} [\cite{conn-med}] \label{cor:mtr-equil} 
The metric triangles of meshed graphs are equilateral.
\end{corollary}  
\begin{corollary} \label{cor:metric-tr} 
Let $G$ be a meshed graph and $uvw$ be a metric triangle of $G$ of size $k$. For every shortest path $P(u,v)$ connecting $u$ and $v$ such that $P(u,v)\subset B(w,k)$, all vertices of $P(u,v)$ are at the same distance $k$ from $w$. 
\end{corollary}

Next, we consider meshed graphs satisfying  a ($\lambda,\mu$)-bow metric for some $\lambda\ge 0$ and $\mu\ge 0$. 

A subset $S$ of a geodesic metric space or a graph is \emph{convex} if for all $x,y \in S$ the metric interval $I(x,y)$ is contained in $S$. This notion was extended by Gromov~\cite{Gromov1987} as follows:
for $\epsilon \geq 0$, a subset $S$ of a geodesic metric space or a graph is called \emph{$\epsilon$-quasiconvex} 
if for all $x,y \in S$ the metric interval $I(x,y)$ is contained in the ball $B(S,\epsilon)$. $S$ is said to be \emph{quasiconvex} if there is a constant $\epsilon \geq 0$ such that $S$ is $\epsilon$-quasiconvex.
Quasiconvexity plays an important role in the study of hyperbolic and cubical groups,
and hyperbolic graphs contain an abundance of quasiconvex sets~\cite{Chepoi:2017:CCI:3039686.3039835}.

\begin{lemma} \label{lm:balls-conv} 
Let $G$ be a meshed graph satisfying ($\lambda,\mu$)-bow metric. Then, all balls of $G$ are $k$-quasiconvex for some $k\le \max\{\lambda,\mu/2\}$.
\end{lemma}

\begin{proof} Consider in a meshed graph $G$ a ball $B(w,r)$ and two arbitrary vertices $u,v\in B(w,r)$. Consider a vertex $y$ of $I(u,v)$ such that $d(y,w)=\max\{d(x,w): x\in I(u,v)\}$ (i.e., $y$ is a  vertex of  $I(u,v)$ most distant from $w$). Let $k:=d(y,w)-r>0$. We want to show that $k\le \max\{\lambda,\mu/2\}$. Consider a shortest path $P(v,y)$ between $v$ and $y$ whose sum $\sigma:=\sum_{x\in P(v,y)} d(x,w)$ is smallest. Let $v'y'$ be the first edge of $P(v,y)$ (when moving from $v$ to $y$) such that $d(w,v')=r=d(w,y')-1$ (i.e., the first edge leaving the ball $B(w,r)$). We claim that the subpath $P':=(v'=x_0,y'=x_1,\dots, x_i,\dots, x_{\ell}=y)$ of $P(v,y)$ is monotone in the sense that $d(w,x_i)<d(w,x_{i+1})$ for each $i$, $0\le i<\ell$ (in other words, we claim that $\ell=k$). Assume this is not the case and consider the first edge $x_ix_{i+1}$ of $P'$ (when moving from $v'$ to $y$ along $P'$) such that $d(x_i,w)\ge d(x_{i+1},w)$.  By the meshedness of $G$ applied to vertices $x_{i-1},x_i,x_{i+1}$ and $w$, we will find a new vertex $z$ which is adjacent to $x_{i-1},x_{i+1}$ and at distance at most $d(w,x_{i-1})$ from $w$.  The latter contradicts our assumption that  $P(v,y)$ had smallest sum $\sigma$; replacing $x_i$ with $z$ in $P(v,y)$  will produce a shortest path between $v$ and $y$ with a smaller sum. Thus, $\ell=k$ must hold.  

Now, if $d(v',y)=\ell=k$ is greater than $\lambda$, we can apply ($\lambda,\mu$)-bow metric to $y\in I(v',u)$, $v'\in I(w,y)$ and get $r\ge d(w,u)\ge d(w,v')+d(v',u)-\mu= r+d(v',y)+d(y,u)-\mu\ge r+2k-\mu$ (as $d(v',y)=k\le d(y,u)$). That is, if $k> \lambda$, then $k\le \mu/2$. Consequently, $k\le  \max\{\lambda,\mu/2\}$.
\end{proof}

\begin{lemma} \label{lm:metric-tr-size} 
Let $G$ be a meshed graph satisfying ($\lambda,\mu$)-bow metric. Then, all metric triangles of $G$ are equilateral of size at most $\lambda+2\mu+1$. 
\end{lemma}

\begin{proof}
It is known \cite{meshed} that all metric triangles of a meshed graph are equilateral (see also Corollary \ref{cor:mtr-equil}).  Consider an equilateral metric triangle $uvw$ of size $k:=d(u,v)$ and shortest paths $P(w,u)$ and $P(w,v)$ such that $d(u,y)=k$, for all $y\in P(w,v)$, and $d(v,x)=k$, for all $x\in P(w,u)$ (such paths exist by Theorem \ref{th:metric-tr}). Assuming $k\ge \lambda+\mu+1$, pick vertices $x\in P(w,u)$ and $y\in P(w,v)$ such that $d(w,x)=d(w,y)=\lambda+\mu+1$. We know $d(v,x)=d(u,y)=k$. 
Let $x'y'w'$ be a quasi-median of $x,y$ and $w$. Let $\ell:=d(w',x')=d(x',y')=d(x',w')$ (recall that $x'y'w'$ is an equilateral metric triangle). Since  $uvw$ is a metric triangle, necessarily, $w=w'$. Furthermore, by Lemma \ref{lm:isom-balls}  and Corollary \ref{cor:metric-tr},  $d(x',v)=d(y',u)=k$. See Fig. \ref{fig:three} 
for an illustration. 

  \begin{figure}[htb]
    \begin{center} \vspace*{-12mm}
      \begin{minipage}[b]{15cm}
        \begin{center} 
          \hspace*{27mm}
          \includegraphics[height=16cm]{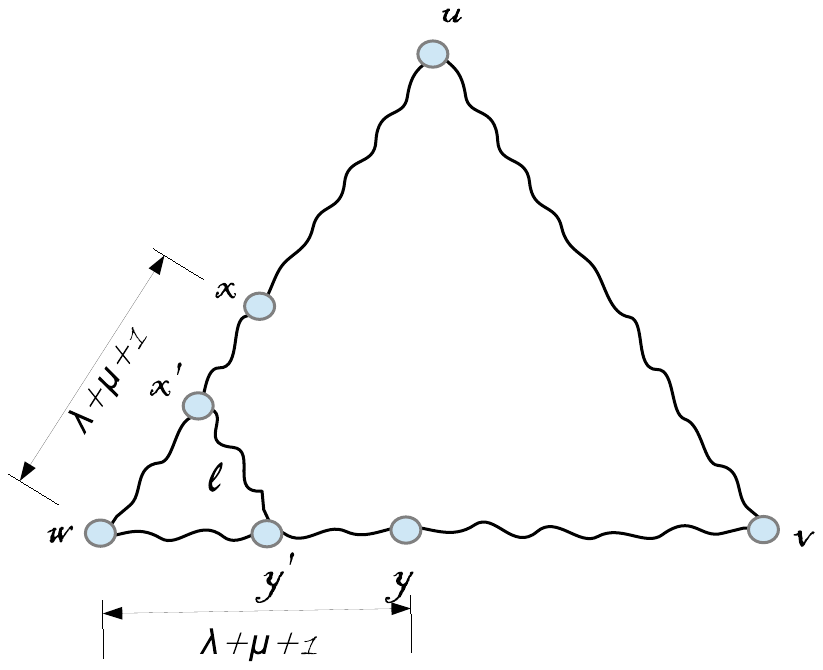}
        \end{center} \vspace*{-86mm}
        \caption{\label{fig:three} An illustration to the proof of Lemma \ref{lm:metric-tr-size}.} %
      \end{minipage}
    \end{center}
   \vspace*{-1mm}
  \end{figure}

First, assume $\ell\ge \lambda+1$. We have $y'\in I(v,x')$ and $x'\in I(u,y')$ since $k=d(v,x')\le d(v,y')+d(y',x')=d(v,y')+\ell=d(v,y')+d(y',w)=d(v,w)=k$ and, similarly, $k=d(u,y')=d(u,x')+d(x',y')$. Applying $(\lambda,\mu)$-bow metric, we get $k=d(v,u)\ge d(v,x')+d(x',u)-\mu =k+k-\ell-\mu$, i.e, $k\le \ell+\mu= d(w,x')+\mu\le d(w,x)+\mu=\lambda+\mu+1+\mu=\lambda+2\mu+1$. 

Assume now $\ell< \lambda+1$. We have  $x'\in I(x,y)$ by the definition of a quasi-median. Additionally, $y\in I(v,x')$ and $d(y,x')=d(y,w)=\lambda+\mu+1$ hold since $k=d(v,x')\le d(v,y)+d(y,x')=d(v,y)+d(y,y')+d(y',x')=d(v,y)+d(y,y')+\ell=d(v,y)+d(y,y')+d(y',w)=d(v,y)+d(y,w)=d(v,w)=k$ and the first inequality must indeed be an equality. 
Applying $(\lambda,\mu)$-bow metric, we get $k=d(v,x)\ge d(v,x')+d(x',x)-\mu=d(v,x')+d(w,x)-\ell-\mu =k+\lambda+\mu+1-\ell-\mu$. That is, $\ell\ge \lambda+1$, and a contradiction with $\ell< \lambda+1$ arises.   
\end{proof}

By Proposition  \ref{prop:bow-int-thin} and Lemma  \ref{lm:metric-tr-size}, 
the intervals of $G$ are $(3\lambda+2\mu+1)$-thin. 

\begin{corollary} \label{cor:leanness} 
Let $G$ be a meshed graph satisfying ($\lambda,\mu$)-bow metric. Then, the intervals of $G$ are $(3\lambda+2\mu+1)$-thin. 
\end{corollary}

By  Lemma  \ref{lm:metric-tr-size} and Corollary \ref{cor:leanness}, we can apply Corollary \ref{cor:delta-hyp-2nd-equidist} and Proposition~\ref{prop:delta-hyp-1st} with $q=\lambda+2\mu+1$ and $p=3\lambda+2\mu+1$ to obtain that meshed graphs satisfying a ($\lambda,\mu$)-bow metric are hyperbolic. Using also Proposition~\ref{prop:sl-vs-hb}, we can state the following.

 \begin{theorem} \label{th:meshed-hyperb} 
For meshed graphs satisfying ($\lambda,\mu$)-bow metric, the following inequalities hold for their hyperbolicity $\delta$ and their slimness $\varsigma$:
\begin{align*}
   \delta\le & ~2(\lambda+\mu)+1, \\
\varsigma\le\ &
  \left\{\begin{array}{l}
                 \frac{13}{2}\lambda+5\mu+\frac{5}{2},   \mbox{~~~~~~~ when } \lambda\le 2(\mu+1)\\
                6(\lambda+\mu)+\frac{7}{2},  \mbox{~~~~~ otherwise. }
    \end{array}\right.
     \end{align*}
\end{theorem}

As weakly modular graphs are meshed, we get the following immediate corollary. 

\begin{corollary} \label{cor:wmg-hyperb} 
For weakly modular graphs satisfying a ($\lambda,\mu$)-bow metric,  the following inequalities hold for their hyperbolicity $\delta$ and their slimness $\varsigma$:
\begin{align*}
   \delta\le & ~2(\lambda+\mu)+1, \\
\varsigma\le\ &
  \left\{\begin{array}{l}
    \frac{13}{2}\lambda+5\mu+\frac{5}{2},   \mbox{~~~~~~~ when } \lambda\le 2(\mu+1)\\
                6(\lambda+\mu)+\frac{7}{2},  \mbox{~~~~~ otherwise. }
    \end{array}\right.
     \end{align*}
%
\end{corollary}

In the next subsections, we will improve bounds on $\delta$ and $\varsigma$ for various subclasses of weakly modular graphs and their relatives.

\subsection{CB-graphs}

As we mentioned earlier, all metric triangles in weakly modular graphs are strongly equilateral.  

\begin{lemma} [\cite{Ch1989-m-tr}]\label{lem:wmg-metr-tr}  A graph $G$ is weakly
modular if and only if all metric triangles of $G$  are strongly equilateral. 
\end{lemma}

A graph $G$ is called a CB-graph if all its balls are convex. The graphs with convex balls have been introduced and characterized in \cite{bridged-FJ,bridged-SCh} as graphs without embedded isometric cycles of lengths different from 3 and 5 and in which for all pairs of vertices $u$ and $v$, all neighbors of $v$ lying
on a shortest $(u,v)$-path form a clique. One of their important subclasses is the class of bridged graphs:
these are the graphs without embedded isometric cycles of length greater than 3 and they are exactly the
graphs in which the balls around convex sets are convex \cite{bridged-FJ,bridged-SCh}.  A comprehensive investigation of CB-graphs was recently done in \cite{CB-graphs2023}. Among other very interesting results, it was shown that in CB-graphs, metric triangles behave quite similarly to metric triangles
in weakly modular graphs.  

\begin{lemma} [\cite{CB-graphs2023}] \label{lem:CB-metr-tr} Every metric triangle of a CB-graph is either  strongly equilateral or has type $(2,2,1)$. 
\end{lemma}

To show that all CB-graphs satisfying a ($\lambda,\mu$)-bow metric are hyperbolic, we first prove the following variant of Lemma  \ref{lm:metric-tr-size}. 

\begin{lemma} \label{lm:wmg+CB--metric-tr-size} 
Let $G$ be a graph satisfying a ($\lambda,\mu$)-bow metric. If all metric triangles of $G$ with maximum
side-length larger than $\psi$ are strongly equilateral, then  all metric triangles of $G$ have maximum
side-length at most $\max\{\psi,\lambda+2\mu+1\}$. 
\end{lemma}

\begin{proof}
Consider a metric triangle $uvw$ in $G$ and assume that its maximum side-length is larger than $\psi$. Then, it is a strongly equilateral metric triangle  of size $k:=d(u,v)>\psi$. Consider arbitrary shortest paths $P(w,u)$ and $P(w,v)$. Since $uvw$ is a strongly equilateral metric triangle, we have $d(u,y)=k$, for all $y\in I(w,v)\supseteq P(w,v)$, and $d(v,x)=k$, for all $x\in I(w,u)\supseteq P(w,u)$. Assuming $k\ge \lambda+\mu+1$, pick vertices $x\in P(w,u)$ and $y\in P(w,v)$ such that $d(w,x)=d(w,y)=\lambda+\mu+1$. We know $d(v,x)=d(u,y)=k$. 
Let $x'y'w'$ be a quasi-median of $x,y$ and $w$. Since  $uvw$ is a metric triangle, necessarily, $w=w'$. We also know that $d(v,x')=d(u,y')=k$ (as $x'\in I(w,x)\subseteq I(w,u)$ and $y'\in I(w,y)\subseteq I(w,v)$). 
But then, $k=d(v,w)=d(v,y)+d(y,y')+d(y',w)$ and $k=d(v,x')=d(v,y)+d(y,y')+d(y',x')$ imply $d(y',x')=d(y',w)$. Similarly, $d(y',x')=d(x',w)$ must hold. Hence, $x'y'w$ is an equilateral metric triangle. Let $\ell:=d(w,x')=d(x',y')=d(x',w)$.  See again Fig. \ref{fig:three} 
for an illustration. 

We have  $x'\in I(x,y)$ by the definition of a quasi-median. Additionally, $d(y,x')=d(y,w)=\lambda+\mu+1$ and $y\in I(v,x')$ hold since $k=d(v,x')\le d(v,y)+d(y,x')=d(v,y)+d(y,y')+d(y',x')=d(v,y)+d(y,y')+\ell=d(v,y)+d(y,y')+d(y',w)=d(v,y)+d(y,w)=d(v,w)=k$. 
Applying $(\lambda,\mu)$-bow metric, we get $k=d(v,x)\ge d(v,x')+d(x',x)-\mu=d(v,x')+d(w,x)-\ell-\mu =k+\lambda+\mu+1-\ell-\mu$. That is, $\ell\ge \lambda+1$.  
Now, we have $y'\in I(v,x')$ and $x'\in I(u,y')$ since $k=d(v,x')\le d(v,y')+d(y',x')=d(v,y')+\ell=d(v,y')+d(y',w)=d(v,w)=k$ and, similarly, $k=d(u,y')=d(u,x')+d(x',y')$. Applying $(\lambda,\mu)$-bow metric, we get $k=d(v,u)\ge d(v,x')+d(x',u)-\mu =k+k-\ell-\mu$, i.e, $k\le \ell+\mu= d(w,x')+\mu\le d(w,x)+\mu=\lambda+\mu+1+\mu=\lambda+2\mu+1$. 
 \end{proof}

For CB-graphs, we have the following corollary. 

\begin{corollary} \label{cor:CB-metr-triangles} 
Let $G$ be a CB-graph satisfying a ($\lambda,\mu$)-bow metric. Then, all metric triangles of $G$ have maximum
side-length at most $\max\{2,\lambda+2\mu+1\}$. 
\end{corollary}

Proposition \ref{prop:bow-int-thin} can be improved for CB-graphs. 

\begin{lemma}\label{lm:CB-leanness}
Let $G$ be a CB-graph satisfying a ($\lambda,\mu$)-bow metric. Then, the intervals of $G$ are $\max\{2,\lambda+2\mu+1\}$-thin. 
\end{lemma}
\begin{proof} 
Consider arbitrary vertices $u,v$ in $G$ and arbitrary vertices $x,y\in S_k(u,v)$, where $1\le k \le d(u,v)-1$. 
Let $x'y'v'$ be a quasi-median of $x,y,v$. Since $x'y'v'$ is a metric triangle of $G$, we have $d(x',y')\le  \max\{2,\lambda+2\mu+1\}$. If $d(x,x')>0$, then a shortest path between $x$ and $y$ passing through $x'$ and $y'$ violates the convexity of ball $B(u,k)$ as we then have $x,y\in B(u,k)$ and $x'\notin B(u,k)$. 
Hence, $x=x'$ (and, similarly, $y=y'$) must hold, giving $d(x,y)=d(x',y')\le  \max\{2,\lambda+2\mu+1\}$.
\end{proof}

By Proposition  \ref{prop:delta-hyp-1st}, Proposition \ref{prop:delta-hyp-2nd}, Corollary \ref{cor:CB-metr-triangles}, and Lemma  \ref{lm:CB-leanness},  CB-graphs satisfying a ($\lambda,\mu$)-bow metric are hyperbolic.   

 \begin{theorem} \label{th:CB-hyperb}  
CB-graphs satisfying a ($\lambda,\mu$)-bow metric have slimness at most $\frac{5}{2}\max\{2,\lambda+2\mu+1\}$ and 
are $\frac{5}{2}\max\{2,\lambda+2\mu+1\}$-hyperbolic. 
\end{theorem}

\subsection{Modular graphs, Pseudo-modular graphs, Median graphs, and Helly graphs}\label{subsec:MedianHellyETC}
A graph is called {\em median} if $|I(u, v) \cap I(v, w) \cap I(w, v)| = 1$ for every triplet $u, v, w$ of vertices, that is, every triplet of vertices has a unique median (i.e., quasi-median of size 0). Median graphs can be characterized in
several different ways and they play an important role in geometric group theory, concurrency, as well as in combinatorics (see survey ~\cite{BaCh-survey}). 
A graph is called {\em modular} if $I(u, v) \cap I(v, w) \cap I(w, u) \neq\emptyset$ for every triplet $u,v, w$ of vertices, i.e.,
every triplet of vertices admits a (not necessarily unique) median. Clearly median graphs are modular.
Modular graphs are exactly the weakly modular graphs in which all metric triangles of G have size 0. 
A graph $G$ is called {\em pseudo-modular} if any three pairwise intersecting balls of $G$ have a nonempty common intersection \cite{ps-modular}. This condition easily implies both the triangle and quadrangle conditions, and thus
pseudo-modular graphs are weakly modular. In fact, pseudo-modular graphs are quite specific weakly
modular graphs: from the definition also follows that all metric triangles of pseudo-modular graphs have
size 0 or 1, i.e., each metric triangle is either a single vertex or is a triangle of $G$.
Recall also that a graph G is a  {\em Helly} graph if the family of balls of $G$ has the Helly property, that is, every 
collection of pairwise intersecting balls of $G$ has a nonempty common intersection. From the definition
it immediately follows that Helly graphs are pseudo-modular. Helly graphs are the discrete analogues
of hyperconvex spaces: namely, the requirement that radii of balls are from the nonnegative reals is
modified by replacing the reals by the integers. In perfect analogy with hyperconvexity, there is a close
relationship between Helly graphs and absolute retracts: absolute retracts and Helly graphs are the same \cite{BanPesch,retracts}. In particular, for any graph $G$ there exists a smallest Helly graph comprising $G$ as an isometric
subgraph. 

It is known that the hyperbolicity of a median graph as well as of a Helly graph is governed by the thinness of its intervals \cite{CCHO-WMG2020,delta-hyp-1st,obstructions}. Our Corollary \ref{cor:delta-hyp-2nd-equidist} generalizes those results. Recall that for every $\delta$-hyperbolic graph, $\tau(G)\le 2\delta$.  

\begin{proposition}  \label{prop:Helly-hyp-th} 
If the intervals of a pseudo-modular $($in particular, of a Helly graph$)$ $G$ are $p$-thin, then $G$ is $\delta$-hyperbolic for $\frac{p}{2}\le \delta\le \frac{p+1}{2}$.  
\end{proposition}

\begin{proposition}  \label{prop:Median-hyp-th} 
If the intervals of a modular $($in particular, of a median graph$)$ $G$ are $p$-thin, then $G$ is $\delta$-hyperbolic for $\delta=\frac{p}{2}$.  
\end{proposition}

Since all metric triangles of modular (and, hence, of median) graphs have size 0 and  all metric triangles of pseudo-modular (and, hence, of Helly) graphs have size 0 or 1,  
from Proposition \ref{prop:delta-hyp-1st}, Corollary \ref{cor:delta-hyp-2nd-equidist}, Proposition \ref{prop:bow-int-thin},  Proposition \ref{prop:Helly-hyp-th} and  Proposition \ref{prop:Median-hyp-th}, we have immediate corollaries.  
 
\begin{corollary} \label{cor:Helly} Pseudo-modular $($and, hence, Helly$)$ graphs satisfying a ($\lambda,\mu$)-bow metric are $\delta$-hyperbolic for $\delta\le \max\{\frac{\mu+1}{2},\lambda+1\}$.  Their slimness $\varsigma$ is bounded as follows:  
\begin{align*}
\varsigma\le\ &
  \left\{\begin{array}{l}
                 \frac{3}{2}\max\{\mu,2\lambda+1\}+2,   \mbox{ when } \lambda\ge 2 \mbox{~ or } \mu\ge 4\\
                2\max\{\mu,2\lambda+1\},  \mbox{~~~~~ otherwise. }
    \end{array}\right.
     \end{align*}
\end{corollary}

\begin{corollary} \label{cor:median} Modular $($and, hence, median$)$ 
      graphs satisfying a ($\lambda,\mu$)-bow metric are $\delta$-hyperbolic for $\delta\le \max\{\frac{\mu}{2},\lambda\}$.  Their slimness $\varsigma$ is at most   $\lceil\frac{3}{2}\max\{\mu,2\lambda\}\rceil$. 
\end{corollary}

The results of Corollary \ref{cor:Helly} and Corollary \ref{cor:median} show that the conjecture from \cite{alpha-hyperb} that {\em every  $\alpha_i$-metric graph is $\delta$-hyperbolic for some $\delta\le \frac{i+1}{2}$} is true within pseudo modular graphs (and, hence, for Helly graphs, modular graphs, and median graphs).  


\subsection{Bipartite graphs}
 Now we consider bipartite graphs. First, we show that bipartite graphs satisfying a ($1,\mu$)-bow metric are $\delta$-hyperbolic for some $\delta \le 3(\mu+1)/2+4$. Note that any graph (not necessarily bipartite) satisfying a ($0,\mu$)-bow metric (as they are $\alpha_{\mu}$-metric) are $\delta$-hyperbolic for some $\delta \le 3(\mu+1)/2$~\cite{alpha-hyperb}. Then, we conclude this subsection by showing that, if our conjecture (that ($\lambda,\mu$)-bow metric implies hyperbolicity) is true for all bipartite graphs or for all line graphs of bipartite graphs, then it is true for all graphs.  

We will need the following interesting lemma. Recall that the line graph $L(G)$ of a graph $G$ is a graph such that each vertex of $L(G)$ represents an edge of $G$ and two vertices of $L(G)$ are adjacent if and only if their corresponding edges in $G$ share a common endpoint (i.e., they are incident) in $G$.

\begin{lemma}\label{lem:bip-decrease-lambda}
    If $G$ is a bipartite graph satisfying a ($\lambda,\mu$)-bow metric for some $\lambda > 0$, then its line graph $L(G)$ satisfies $(\lambda-1,\mu+2)$-bow metric.
\end{lemma}
\begin{proof}
    Consider four vertices $e_u,e_x,e_y,e_v$ of $L(G)$ such that $d_{L(G)}(e_x,e_y) = \lambda$ and $e_x \in I_{L(G)}(e_u,e_y)$, $e_y \in I_{L(G)}(e_x,e_v)$. 
    Fix a shortest path $P(e_u,e_y) = (e_u=e_0,e_1,\ldots,e_p=e_y)$ of $L(G)$ such that $e_{p-\lambda}=e_x$.
    For the unique end-vertices $u,y'$ in the intersections $e_u \cap e_1$, $e_{p-1}\cap e_y$ respectively, the internal edges $e_1,e_2,\ldots,e_{p-1}$ induce a shortest $(u,y')$-path of $G$.
    Write $e_y = y'y$.
    If $d_G(u,y) = p-2$, then we could replace in $P(e_u,e_y)$ the internal sequence $e_1,e_2,\ldots,e_{p-1}$ by the $p-2$ edges of a shortest $(u,y)$-path of $G$.
    The latter would contradict our assumption that $P(e_u,e_y)$ is a shortest path.
    Therefore, $d_G(u,y) \ge p-1$, which implies $d_G(u,y) = p$ because $G$ is bipartite.
    In particular, edges $e_x,e_y$ lie on a shortest $(u,y)$-path of $G$.    
    We obtain similarly the existence of end-vertices $v$ of $e_v$ and $x$ of $e_x$ such that $e_y,e_x$ lie on a shortest $(v,x)$-path of $G$.
    Doing so, $x \in I(u,y)$ and $y \in I(x,v)$.
    Furthermore, $y$ cannot be a closest to $x$ end-vertex of $e_y$, that is because edges $e_x,e_y$ lie on some shortest $(u,y)$-path of $G$.
    By symmetry, $x$ cannot be a closest to $y$ end-vertex of $e_x$.
    Altogether, it implies that $e_x,e_y$ are the terminal edges of some shortest $(x,y)$-path of $G$.
    Then, $d_G(x,y) \ge d_{L(G)}(e_x,e_y) + 1 = \lambda+1$.
    Since $G$ is $(\lambda,\mu)$-bow metric, $d_G(u,v) \ge d_G(u,x) + d_G(x,y) + d_G(y,v) - \mu = (d_{L(G)}(e_u,e_x)-1)+(\lambda+1)+(d_{L(G)}(e_y,e_v)-1) - \mu = d_{L(G)}(e_u,e_x) + \lambda + d_{L(G)}(e_y,e_v) - (\mu+1)$.
    Finally, $d_{L(G)}(e_u,e_v) \ge d_G(u,v)-1 \ge d_{L(G)}(e_u,e_x) + \lambda + d_{L(G)}(e_y,e_v) - (\mu+2)$.
\end{proof}

As a consequence, we get the following corollary. 

\begin{corollary}
    Bipartite graphs satisfying a ($1,\mu$)-bow metric are $\delta$-hyperbolic for some $\delta \le 3(\mu+1)/2+4$. 
\end{corollary}
\begin{proof}
    Let $G$ be any bipartite ($1,\mu$)-bow metric graph.
    By Lemma~\ref{lem:bip-decrease-lambda}, its line graph $L(G)$ is $\alpha_{\mu+2}$-metric.
    Then, $L(G)$ is $\delta'$-hyperbolic, for some $\delta' \le 3(\mu+1)/2 + 3$~\cite{alpha-hyperb}.
    The latter implies that $G$ is $\delta$-hyperbolic, for some $\delta \le \delta'+1 \le 3(\mu+1)/2 + 4$~\cite{CoDu16}.
\end{proof}

The case of bipartite graphs satisfying a ($\lambda,\mu$)-bow metric
for $\lambda\ge 2$ is open. In fact, as it follows from the proof of Theorem \ref{th:reduce} below, if our conjecture is true for all bipartite graphs, then it is true for all graphs as well. 

We can show the following interesting result. It also reduces our conjecture from general graphs to the line graphs of bipartite graphs. 

\begin{theorem} \label{th:reduce}  
A graphs $G$ satisfying a ($\lambda,\mu$)-bow metric has hyperbolicity at most $f(\lambda,\mu)$ if and only if the line graph $L(H)$ of its $1$-subdivision $H$ has hyperbolicity at most $O(f(\lambda,\mu))$.
\end{theorem}
\begin{proof} Let $G$ be a graph satisfying a ($\lambda,\mu$)-bow metric.  
    By Lemma \ref{lm:1-subdiv}, the $1$-subdivision $H$ of $G$ satisfies ($2\lambda+2,2\mu+2$)-bow metric. By Lemma \ref{lem:bip-decrease-lambda}, the line graph $L(H)$ satisfies $(2\lambda+1,2\mu+4)$-bow metric. 
    
    Now, if the hyperbolicity of $L(H)$ is at most $\delta''=g(\lambda,\mu)$ (for some function $g(\cdot,\cdot)$), then $H$ is $\delta'$-hyperbolic for some $\delta'\le \delta''+1$ (see ~\cite[Theorem 6]{CoDu16}) and, hence,  $G$ is $\delta$-hyperbolic for some $\delta\le \delta'/2$ (see ~\cite[page 194]{CoDu16}). 
\end{proof}

\begin{corollary}
    If our conjecture (that ($\lambda,\mu$)-bow metric implies hyperbolicity) is true for all bipartite graphs or for all line graphs of bipartite graphs, then it is true for all graphs.  
\end{corollary}

\section{Conclusion}
We conjectured that, in graphs, a ($\lambda,\mu$)-bow metric implies hyperbolicity and showed that our conjecture is true for several large families of graphs. If there is a counterexample (non-hyperbolic) graph to our conjecture, it must satisfy a ($\lambda,\mu$)-bow metric but have unbounded side lengths of metric triangles  and an unbounded interval thinness. The question whether ($\lambda,\mu$)-bow metric implies hyperbolicity in all graphs remains  as a main open question. It is sufficient to answer it for bipartite graphs or for line graphs of bipartite graphs.  We demonstrated also that many known in literature graph classes satisfy a $(\lambda,\mu)$-bow metric for some bounded values of $\lambda$ and $\mu$. Another remaining open question is whether some other  interesting families of graphs satisfy $(\lambda,\mu)$-bow metrics for some small values of $\lambda$ and $\mu$.


\end{document}